\documentclass [11pt] {article}
\usepackage{amssymb}
\usepackage{times}
\usepackage{graphicx}
\usepackage{amsthm,amsmath,amsfonts,color}
\usepackage{hyperref}

\newcommand{\aff}{\mathrm{Trans}}
\newcommand{\rk}{\mathrm{rk}}

\newcommand{\Hom}{\mathrm{Hom}}

\newtheorem{lemma}{Lemma}[section]
\newtheorem{corollary}{Corollary}[section]
\newtheorem{theorem}{Theorem}

\newtheorem{proposition}{Proposition}[section]
\newtheorem{definition}{Definition}[section]
\theoremstyle{definition}
\newtheorem{remark}{Remark}[section]
\newtheorem*{claim}{Claim}

\begin{document}

\title{Nilspaces, nilmanifolds and their morphisms}
\author{O.Antolin Camarena~~~~~ Bal\'azs Szegedy}

\maketitle

\abstract{Recent developments in ergodic theory, additive
  combinatorics, higher order Fourier analysis and number theory give
  a central role to a class of algebraic structures called {\it
    nilmanifolds}. In the present paper we continue a program started
  by Host and Kra. We study {\it nilspaces} that are structures
  satisfying a variant of the Host-Kra axiom system for parallelepiped
  structures. We give a detailed structural analysis of abstract and
  compact topological nilspaces. Among various results it will be
  proved that compact nilspaces are inverse limits of finite
  dimensional ones. Then we show that finite dimensional compact
  connected nilspaces are nilmanifolds with an extra algebraic
  structure determined by a filtration on the corresponding Lie
  group. The theory of compact nilspaces is a generalization of the
  theory of compact abelian groups. This paper is the main algebraic
  tool in the second authors approach to Gowers's uniformity norms and
  higher order Fourier analysis.}

\tableofcontents

\section{Introduction}

We start with the formal definition of $k$-step nilmanifolds.

\begin{definition} Let $L$ be a $k$-nilpotent Lie group. This means that the $k$-fold iterated commutator $$[\dots[[L,L],L],L\dots]$$ is trivial. Let $\Gamma$ be a co-compact subgroup in $L$. The left coset space $N=L/\Gamma$ is a compact topological space which is called a $k$-step {\bf nilmanifold}.
\end{definition}

Nilmanifolds were first introduced and studied by Mal'cev \cite{Malc} in 1951. He proved many crucial facts which can be also found in the book \cite{Rag}. Nilmanifolds are interesting from a purely geometric point of view \cite{Grom},\cite{Miln}. However recent development \cite{GrTao},\cite{HKr},\cite{Zieg} shows their important role in ergodic theory and additive combinatorics. In addition to this a spectacular application of nilmanifolds (more precisely nil-sequences) in number theory was presented by Green and Tao in \cite{GrTao2},\cite{GTZ}. 

\medskip

The main motivation for this paper comes from higher order Fourier analysis. Let $f$ be a bounded measurable function on a compact abelian group $A$. We denote by $\Delta_t f$ the function $x\mapsto f(x)\overline{f(x+t)}$. The $U_k$ uniformity norm of $f$ introduced by Gowers \cite{Gow},\cite{Gow2} is defined by
$$\|f\|_{U_k}=\bigl(\mathbb{E}_{t_1,t_2,\dots,t_k}\Delta_{t_1,t_2,\dots,t_k}(f)\Bigr)^{2^{-k}}.$$

In particular it can be computed that $$\|f\|_{U_2}=\Bigl(\sum_{\chi\in\hat{A}}|(f,\chi)|^4\Bigr)^{1/4}$$ where $\hat{A}$ is the dual group of $A$.
This formula explains the behavior of the $U_2$ norm through ordinary Fourier analysis.

Based on results in ergodic theory \cite{HKr},\cite{Zieg} it is expected that the behavior of the $U_k$ norm is in some sense connected to $k-1$ step nilmanifolds. However to clarify the precise connection (at least in the second author's interpretation) one needs a generalization of $k$-step nilmanifolds that we call $k$-step nilspaces. 
These objects, with a different set of axioms, were first introduced by Host and Kra \cite{HKr} and studied for $k=2$. Their pioneering work paved the way for our theory developed in this paper. 

Before giving the precise definition of $k$-step nilspaces we give a list of motivations and reasons to generalize nilmanifolds.

\begin{enumerate}
\item The structures which naturally arise in ergodic theory are not nilmanifolds but inverse limits of them.
\item A $k$-step nilspace (even if it is a nilmanifold topologically) has an extra algebraic structure which is needed in Higher order Fourier analysis.
\item In higher order Fourier analysis it will be convenient to study morphisms between nilmanifolds and nilspaces. It turns out that nilspaces are more natural for this purpose than nilmanifolds.
\item To study Gowers norms of functions on abelian groups with many bounded order elements nilmanifolds are not enough.
\item Nilspaces are directly defined through a simple set of axioms. This helps to separate the algebraic and analytic difficulty in Higher order Fourier analysis.
\item Gowers norms can be naturally defined for functions on compact nilspaces. This means that the notion of Higher order Fourier analysis naturally extends to them.
\item Related to the so-called limit theory for graphs and hypergraphs, interesting limit notions can be defined for functions on abelian groups. The limit objects are functions on nilspaces.
\end{enumerate}

The axiom system of nilspaces is a variant of Host-Kra's axiom system
\cite{HKr2} for parallelepiped structures.  Roughly speaking, a
nilspace is a structure in which cubes of every dimension are defined
and they behave in a very similar way as in an abelian group.  An
abstract $n$-dimensional cube is the set $\{0,1\}^n$.  A cube morphism
$\phi:\{0,1\}^n\rightarrow\{0,1\}^m$ is a map which extends to an
affine homomorphism (a homomorphism plus a shift) from $\mathbb{Z}^n$
to $\mathbb{Z}^m$. Equivalently, cube morphisms are those functions $f
: \{0,1\}^n \to \{0,1\}^m$ that can be written as $f(x_1, \ldots, x_n)
= (y_1, \ldots, y_m)$ where each $y_i$ is either $0$, $1$, $x_j$ or
$1-x_j$ for some $j$ (depending on $i$).  A nilspace is a set $N$
together with sets $C^n(N)\subseteq N^{\{0,1\}^n}$ of $n$ dimensional
cubes $f:\{0,1\}^n\rightarrow N$ for every integer $n\geq 0$ which
satisfy the following three axioms:

\begin{enumerate}
\item {\bf(composition)}: If $\phi:\{0,1\}^n\rightarrow\{0,1\}^m$ is a
  cube morphism and $f\in C^m(N)$ then $f \circ \phi\in C^n(N)$
\item {\bf(ergodicity)}: $C^1(N)=N^{\{0,1\}}.$
\item {\bf(gluing)}: Let $f:\{0,1\}^n\setminus 1^n$ be a map whose restrictions to $n-1$ dimensional faces containing $0^n$ are all cubes. Then $f$ extends to $\{0,1\}^n$ as an element in $C^n(N)$.
\end{enumerate}

We don't always assume the ergodicity axiom. If $N$ is not ergodic but satisfies the weaker condition $C^0(N)=N$ (equivalently: constant maps are morphisms) then it can be decomposed into a disjoint union of ergodic nilspaces.
We say that $N$ is a {\bf $k$-step nilspace} if in the gluing axiom the extension is unique for $n=k+1$.
It is not hard to see that $1$-step nilspaces are affine abelian groups with the usual notion of cubes. A cube $f:\{0,1\}^n\rightarrow A$ in an abelian group $A$ is a map which extends to an affine homomorphism from $\mathbb{Z}^n\rightarrow A$.

If a set $N$ satisfies the first axiom and that $C^0(N)=N$ (but not necessarily the others) then we say that $N$ is a {\bf cubespace}.
A {\bf morphism} $h:N\rightarrow M$ between two cubespaces $N$ and $M$ is a cube preserving map. We require that for every $f\in C^n(N)$ the composition $h\circ f$ is in $C^n(M)$. We denote by $\Hom(N,M)$ the set of morphisms between $N$ and $M$. In particular $C^n(N)=\Hom(\{0,1\}^n,N)$. With this notion we can introduce the categories of cubespaces and nilspaces.

We say that $N$ is a compact nilspace if it is a compact, Hausdorff, second countable topological space and $C^n(N)$ is a closed subset of $N^{\{0,1\}^n}$ for every $n\in\mathbb{N}$.
Morphisms between compact nilspaces are required to be continuous.

\begin{remark}
  Our cubespaces are presheaves on the category of cubes and cube
  morphisms, i.e., contravariant functors from that category to the
  category of sets (the composition axiom encodes the
  functoriality). Morphisms of cubespaces are simply natural
  transformations. Given a category $\mathcal{C}$ of ``geometric
  objects'' such as these discrete cubes, taking presheaves on
  $\mathcal{C}$ is a standard way of building a category of
  ``complexes build from gluing together those geometric objects'',
  and we will indeed think of cubespaces as built out of cubes in this
  way. Cubespaces are not arbitrary presheaves, in that a cube in a
  cubespace is determined by its vertices (while for an arbitrary
  presheaf, distinct cubes could share all of their
  boundary). A cubespace structure on a set $N$ is precisely a
  subfunctor of the functor $\{0,1\}^n \mapsto \{$ all functions
  $\{0,1\}^n \to N\}$. The extension condition defining nilspaces
  among all cubespaces is of the same nature as extension conditions
  traditionally used in presheaf categories, such as the Kan extension
  condition for simplicial sets.
\end{remark}

The present paper consists of two parts. In the first part we study abstract nilspaces and in the second part we study compact nilspaces.
The main topics in abstract nilspaces are the following:

\bigskip

\begin{enumerate}
\item For every natural number $k$ and nilspace $N$ we introduce a unique factor of $N$ which is a $k$-step nilspace. Then we prove basic properties of these factors.
\item We give a structure theorem for $k$-step nilspaces in terms of iterated abelian bundles.
\item We introduce a cohomology theory for extensions of nilspaces.
\item We study a sequence of groups $\aff_i(N)$ (introduced by Host and Kra) acting on a $k$-step nilspace $N$. They form a central series in the $k$-nilpotent group $\aff_1(N)$.
\end{enumerate}

\bigskip

The main topics in compact nilspaces are the following:

\bigskip

\begin{enumerate}
\item We generalize the concept of Haar measure for $k$-step compact nilspaces.
\item We prove a rigidity result for morphisms. This means that almost morphisms into finite dimensional nilspaces can be corrected into precise morphisms.
\item We show that a $k$-step compact nilspace is the inverse limit of finite dimensional ones.
\item We show that there are countably many finite dimensional $k$-step nilspaces up to isomorphism.
\item We show that a finite dimensional compact nilspace consists of connected components that have a nilmanifold structure. In particular connected finite dimensional nilspaces are nilmanifolds with a cubic structure related to a filtration of the corresponding Lie group..
\end{enumerate}

\medskip

\noindent{\bf Acknowledgement:}~~The authors are extremely thankful to Yonatan Gutman who carefully read an earlier version of this paper and had numerous comments and suggestions.

\medskip

\subsection{The role of nilspaces in Higher order Fourier analysis}

\bigskip

This chapter is a short explanation of the papers \cite{Sz0} and \cite{Sz4} which connect nilspaces with higher order Fourier analysis.
The main goal in \cite{Sz0} is to establish structure theorems for functions on compact abelian groups in terms of Gowers's uniformity norms.
To be more precise let $f:A\rightarrow\mathbb{C}$ be a measurable function on the compact abelian group $A$ such that $|f|\leq 1$.
The goal is to decompose $f$ as $$f=f_s+f_e+f_r$$ where $f_s$ is a structured part of bounded complexity, $f_e$ is an error with small $L^2$ norm and $f_r$ is quasi random with very small $U_k$ norm. We will refer to this decomposition as the {\bf $U_k$-regularity lemma.} (We omit here the precise statement) 
The main question is the following:~{\it What kind of structure is encoded in $f_s$?}

It is proved in $\cite{Sz0}$ and with a different method in $\cite{Sz4}$ that $f_s$ is the composition of two functions $\psi:A\rightarrow N$ and $g:N\rightarrow\mathbb{C}$ where $N$ is a compact finite dimensional $k-1$-step nilspace of bounded complexity, $\psi$ is a nilspace morphism and $g$ is Lipschitz with bounded constant.
(We omit here the definition of the complexity of a nilspace.)

The proof of the decomposition theorem is based on a decomposition theorem on ultra product groups. Let ${\bf A}$ be the ultra product of finite (or more generally compact) abelian groups.
One can introduce a natural measure space structure on ${\bf A}$ and a $\sigma$-topology (like a topology but only countable unions of open sets need be open).
A topological factor of ${\bf A}$ is given by a surjective continuous map $f:{\bf A}\rightarrow T$ (called factor map) where $T$ is a separable compact Hausdorff space. (Such a factor can also be viewed as an equivalence relation on ${\bf A}$ whose classes are the fibers of $f$.) Every such factor inherits a cubespace structure from ${\bf A}$ by composing the cubes in ${\bf A}$ with the factor map $f$.
A {\bf nilspace factor} of ${\bf A}$ is a topological factor of ${\bf A}$ whose inherited cubespace structure satisfies the nilspace axioms.

The non-standard $U_k$-regularity lemma is simpler to state than the standard one. It says the following. 

\bigskip

\noindent{\bf Non-standard $U_k$-regularity lemma:}~~{\it Every measurable function $f:{\bf A}\rightarrow\mathbb{C}$ with $\|f\|_\infty\leq\infty$ can be (uniquely) decomposed as $f=f_s+f_r$ where $\|f_r\|_{U_k}=0$ and $f_s$ is Borel measurable in a compact $k-1$ step nilspace factor.}

\bigskip

Note that $U_k$ is only a semi-norm on ${\bf A}$ so it is possible that $f_r$ is not $0$ but $\|f_r\|_{U_k}$ is $0$.
The non-standard $U_k$-regularity lemma implies the ordinary one using the results in this paper and standard techniques.

\subsection{Nilmanifolds as nilspaces}

Let $G$ be an at most $k$-nilpotent group. Let $\{G_i\}_{i=1}^{k+1}$ be a central series with $G_{k+1}=\{1\}$, $G_1=G$ and $[G_i,G_j]\subseteq G_{i+j}$.
We define a cubic structure on $G$ which depends on the given central series.
The set of $n$ dimensional cubes $f:\{0,1\}^n\rightarrow G$ is the smallest set satisfying the following properties.
\begin{enumerate}
\item {\it The constant $1$ map is a cube,}
\item {\it If $f:\{0,1\}^n\rightarrow G$ is a cube and $g\in G_i$ then the function $f'$ obtained from $f$ by multiplying the values on some $(n-i)$-dimensional face from the left by $g$ is a cube.}
\end{enumerate}

This definition builds up cubes by a generating system. However there is another way of describing them through equations.
For every $n$ we introduce an ordering $g_n:\{0,1\}^n\rightarrow\{1,2,\dots,2^n\}$ in the following way.
If $n=1$ then $g_1(0)=1,g_1(1)=2$.
If $n>1$ then $$g_n(a_1,a_2,\dots,a_n)=g_{n-1}(a_1,a_2,\dots,a_{n-1})$$ if $a_n=0$ and
$$g_n(a_1,a_2,\dots,a_n)=2^n+1-g_{n-1}(a_1,a_2,\dots,a_{n-1})$$ if $a_n=1$.
It is clear that (a cyclic version of) this ordering defines a Hamiltonian cycle of the one dimensional skeleton of $\{0,1\}^n$.

\begin{definition}  Let $G$ be a group and $f:\{0,1\}^n\rightarrow G$. We say that $f$ satisfy the {\bf Gray code property} if $$\prod_{i=1}^{2^n}f(g_n^{-1}(i))^{(-1)^i}=1.$$
\end{definition}

A function $f:\{0,1\}^n\rightarrow G$ is a cube if for every $i\in\mathbb{N}$ and $i$-dimensional face $F$ the restriction of $f$ to $F$ satisfies the Gray code property modulo $G_i$. If $i\geq k+1$ then we define $G_1$ to be trivial.
An easy induction shows that cubes in $G$ defined as above are symmetric under the automorphisms of $\{0,1\}^n$.

Assume that $G$ has a transitive action on a set $N$. Then we say that $f:\{0,1\}^n\rightarrow N$ is a cube if
$f(v)=x^{f'(v)}$ where $f':\{0,1\}^n\rightarrow G$ is a cube and $x\in
N$ is a fixed element. It is easy to see that these cubes make $N$
into a $k$-step nilspace.

If $G$ is actually a Lie group and $N$ is a nilmanifold $G / \Gamma$,
we would want the structure above to make $N$ into a \emph{compact}
nilspace, but this is not true without a further assumption on the
central series: to make $C^n(N)$ a closed subset of $N ^ {\{0,1\}^n}$,
we also need that $G_i \cap \Gamma$ is co-compact in $G_i$ for each
$i$. Under that assumption $N$ is indeed a compact nilspace.

\section{Abstract nilspaces}

\subsection{Operations with cubespaces and nilspaces}\label{operations}

Cubespaces and nilspaces are algebraic structures and most of the standard algebraic operations can be defined for them. However, quite interestingly, new operations appear which will be crucial in our proofs.

\medskip

\noindent{\bf Dirct product:}~If $P_1$ and $P_2$ are cubespaces then we define their direct product as the cubespace $P_1\times P_2$ whose cubes are functions $f:\{0,1\}^n\rightarrow P_1\times P_2$ such that the projections $f_1$ and $f_2$ to the direct components are both cubes. It is esay to see that direct products of nilspaces are nilspaces. 

\medskip

\noindent{\bf Sub-cubespace:}~If $P$ is a cubespace than a cubespace defined on $X\subseteq P$ is called sub-cubespace of $P$ if every cube $c:\{0,1\}^n\rightarrow X$ is also a cube in $P$.

\medskip

\noindent{\bf Congruence and Factor:}~Let $P$ be a cubespace and let $\sim$ be an equivalence relation on $N$. By composing cubes $f:\{0,1\}^n\rightarrow P$ in $C^n(P)$ with the projection $P\rightarrow P/\sim$ we obtain an induced cubespace structure on $P/\sim$. If $P$ is a nilspace then we say that $\sim$ is a congruence if the inherited cubespace structure on $P/\sim$ satisfies the nilspace axioms. The nilspace $P/\sim$ is called a factor of $P$. We will see that a stronger notion of factor is more useful in many application. This will be introduced later.

\medskip

\noindent{\bf Arrow space:}~If $N$ is a nilspace then taking direct product with itself yields a nilspace structure on $N\times N$. However there is another interesting cubspace structure on $N\times N$ and we will refer to it as the arrow space. In this construction a map $f_1\times f_2:\{0,1\}^n\rightarrow N\times N$ is a cube if the function $f':\{0,1\}^{n+1}\rightarrow N$ defined by $f'(v,0)=f_1(v),f'(v,1)=f_2(v)$ is a cube in $N$. The arrow space has fewer cubes than the direct product. The arrow space is not necessarily ergodic but lemma \ref{simpglue} implies that it satisfies the gluing axiom and so all its ergodic components are nilspaces.

\medskip

\noindent{\bf Higher degree arrow spaces:}~We will also need the following generalizations of the arrow space. The $i$-th arrow space is a (not necessarily ergodic) nilspace on $N\times N$. Let $f_1,f_2:\{0,1\}^n\rightarrow N$ be two maps. We denote by $(f_1,f_2)_i$ the map $g:\{0,1\}^{n+i}\rightarrow N$ such that $g(v,w)=f_1(v)$ if $w\in\{0,1\}^i\setminus\{1^i\}$ and $g(v,w)=f_2(v)$ if $w=1^i$. If $f:\{0,1\}^n\rightarrow N\times N$ is a single map with components $f_1,f_2$ then we denote by $(f)_i$ the map $(f_1,f_2)_i$.
A map $f:\{0,1\}^n\rightarrow N\times N$ is a cube in the $i$-th arrow space if $(f)_i$ is a cube in $N$.

\medskip

\noindent{\bf The operator $\partial_x$.}~Let $x\in N$ be a fixed element. Then the set $N$ is embedded into the arrow space with the map $y\mapsto(x,y)$. Using this embedding $N$ inherits a new cubespace structure that we call $\partial_x N$.
The operator might turn an ergodic nilspace into a non ergodic nilspace. More generally, let us call a nilspace $N$ {\bf $k$-fold ergodic} if $C^k(N)=N^{\{0,1\}^k}$. It is clear that if $N$ is $k$-fold ergodic then $\partial_x N$ is $k-1$-fold ergodic. 

\subsection{Simplicial nilspaces}

\begin{definition} Let $P$ be a cubespace and $X$ be a sub-cubespace in $P$. We say that $X$ has the {\bf extension property} in $P$ if for every nilspace $N$ and morphism $f:X\rightarrow N$ there is a morphism $f':P\rightarrow N$ with $f=f'|_X$. 
\end{definition}

If $S$ is a finite set and $h$ is a subset of $S$ then we denote by $\{0,1\}^S_h$ the set of vectors supported on $h$ which can be regarded
as the discrete cube of dimension $\left| h \right|$ in the obvious way.

\begin{definition} Let $S$ be a finite set and $H$ be an arbitrary set system in $S$.
 The collection of all cube morphisms $$\{f:\{0,1\}^n\rightarrow \{0,1\}^S_h~|~n\in\mathbb{N},~h\in H\}$$  defines a cubespace structure on $P=\cup_{h\in H}\{0,1\}^S_h$.  Cubespaces arising this way will be called {\bf simplicial}.
\end{definition}

In the above definition $P$ is a sub-cubespace in the full cube $\{0,1\}^S$.
Note that without loss of generality we can assume that $H$ is downwards closed. This means that if $h\in H$ then every subset of $h$ is also in $H$. Such set systems are called simplicial complexes. 

\begin{lemma}[Simplicial gluing]\label{simpglue} Let $S$ be a finite set and $P$ be a simplicial cubespace corresponding to a set system $H$ in $S$. Then $P$ has the extension property in $\{0,1\}^S$.
\end{lemma}

\begin{proof} We assume that $H$ is a simplicial complex. Note that the empty set is always a member in $H$. If $H$ is the full complex of subsets in $S$ then there is nothing to prove. If $H$ is not the full complex then there is a set $h'$ which is not in $H$ but every subset of $h'$ is contained in $H$. Let $H'=H\cup\{h'\}$ be a new simplicial complex. Let $f:P\rightarrow N$ be a morphism into some nilspace $N$. The gluing axiom guarantees that we can extend $f$ to $\cup_ {h\in H'}\{0,1\}^S_h$ with the cubespace structure corresponding to $H'$. By iterating this step we can extend $f$ to the full cube. 
\end{proof}

Let $N$ be a nilspace. We say that two cubes $f_1:\{0,1\}^k\rightarrow N$ and $f_2:\{0,1\}^k\rightarrow N$ in $C^k(N)$ are adjacent if they satisfy that $f_1(v,1)=f_2(v,0)$ for every $v\in\{0,1\}^{k-1}$. For such cubes we define their {\bf concatenation} as the function $f_3:\{0,1\}^k\rightarrow N$ with $f_3(v,0)=f_1(v,0)$ and $f_3(v,1)=f_2(v,1)$ for every $v\in\{0,1\}^{k-1}$. 

\begin{lemma} The concatenation of two adjacent cubes is a cube.
\end{lemma}

\begin{proof} Let $S=\{1,2,\dots,k+1\}$, $h_1=\{1,2,\dots,k\}$,$h_2=\{1,2,\dots,k-1,k+1\}$ and $H=\{h_1,h_2\}$.
Let $P$ be the simplicial cube space corresponding to $H$. Every pair of two adjacent cubes of dimension $k$ in a nilspace $N$ can be represented as a morphism $f$ of $P$ to $N$. By lemma \ref{simpglue}, the morphism $f$ extends to a morphism $f'$ of $\{0,1\}^S$ to $N$. The concatenation of the adjacent pair is the restriction of $f'$ to $C=\{0,1\}^{k-1}\times\{(1,0),(0,1)\}$. On the other hand $C$ is a cube in $\{0,1\}^S$. 
\end{proof}

\subsection{The 3-cubes}\label{threecubes}

In this section we define special cube spaces which will be useful in many calculations. These will simply be $n$-cubes of side length two, divided into
unit cubes. (They are called 3-cubes because they have 3 vertices on each
side). We will typically use them to form new cubes in a nilspace by
gluing together other cubes into a 3-cube and taking the outer vertices,
as justified in Lemma \ref{extvert} below.

Let $T_n=\{-1,0,1\}^n$ together with the following cubespace structure.
For every $v\in\{0,1\}^n$ we define the injective map $\Psi_v:\{0,1\}^n\rightarrow T_n$ by
$$\Psi_v(w_1,w_2,\dots,w_n)_j=(1-2v_j)(1-w_j).$$
We consider the smallest cubespace structure on $T_n$ in which all the maps $\Psi_v$ are cubes
(this just means that the $N$-cubes of $T_n$ are taken to be the maps
$\{0,1\}^N \to T_n$ that factor through the inclusion of some $\Psi_v$).
Note that in terms of the direct product introduced above, $T_n$ is just $(T_1)^n$.
We call $T_n$ the {\bf three cube} of dimension $n$.

In applications we will need various morphisms from three cubes to ordinary cubes. Let $f:\{-1,0,1\}\rightarrow\{0,1\}$ be an arbitrary function. It is clear that $f$ is a morphism of $T_1$ onto the one dimensional cube $\{0,1\}$. Then we have that $f^n:
\{-1,0,1\}^n\rightarrow\{0,1\}^n$ is also a morphism.
Similarly, let $f$ be the function $f(1)=(1,0),f(0)=(0,0),f(-1)=(0,1)$. Then $q=f^n$ is an injective morphism of $T_n$ into the $2n$ dimensional cube $\{0,1\}^{2n}$.
By abusing the notation we will identify $T_n$ with the subset $q(T_n)$ in $\{0,1\}^{2n}$.

The embedding $\omega:\{0,1\}^n\rightarrow T_n$ defined by $\omega(v)=\Psi_v(0^n)$ maps the cube $\{0,1\}^n$ to the set $\{1,-1\}^n$ of ``outer vertices'' of $T_n$.
Note that $\omega$ is not a cube in the above cubespace structure defined on $T_n$.
This fact makes the next lemma useful.

\begin{lemma}\label{extvert} Let $m:T_n\rightarrow N$ be a morphism
  into a nilspace $N$. Then the composition $m \circ \omega$ is in $C^n(N)$.
\end{lemma}

\begin{proof} It is clear that $T_n$ is simplicial in $\{0,1\}^{2n}$ so by lemma \ref{simpglue} the map $m$ extends to $\{0,1\}^{2n}$. On the other hand $\omega$ is a cube morphism of $\{0,1\}^n$ into $\{0,1\}^{2n}$.
\end{proof}

\subsection{Characteristic factors}

In this section we introduce factors of nil-spaces that are crucial building blocks of them.

\begin{definition}\label{simdef} Let $\sim_k$ be the relation defined through the property that $x\sim_k y$ if and only if there are two cubes $c_1,c_2\in C^{k+1}(N)$ such that $c_1(0^{k+1})=x,c_2(0^{k+1})=y$ and $c_1(v)=c_2(v)$ for every element $v\in\{0,1\}^{k+1}\setminus\{0^{k+1}\}$.
\end{definition}

The relation $\sim_k$ is obviously reflexive and symmetric. The next lemma will imply transitivity.

\begin{lemma}\label{sim1} Two elements $x,y\in N$ satisfy $x\sim_k y$ if and only if there is a cube $c\in C^{k+1}(N)$ such that $c(0^{k+1})=y$ and $c(v)=x$ for all $v\in\{0,1\}^{k+1}\setminus\{0^{k+1}\}$.
\end{lemma}

\begin{proof}  Let $c_1,c_2$ be two cubes satisfying the condition in definition \ref{simdef}.
Let us define the map $\phi=f^{k+1}:T_{k+1}\rightarrow\{0,1\}^{k+1}$
on the 3-cube $T_{k+1}$ where $f(-1)=1,f(0)=0,f(1)=1$. We denote by
$g:T_{k+1}\rightarrow N$ the function which is obtained from $c_1
\circ \phi$ by modifying the value on $1^{k+1}$
 from $x$ to $y$.  The condition on $c_1$ and $c_2$ guarantees that $g$ is a morphism.
Using lemma \ref{extvert} we get that $g\circ\omega$ is in $C^{k+1}(N)$.
\end{proof}

\begin{corollary} The relation $\sim_k$ is an equivalence relation for every $k\in\mathbb{N}$ and nilspace $N$.
\end{corollary}

\begin{proof} Assume that in $N$ three elements satisfy $x\sim_k y$ and $y\sim_k z$. then by symmetry and lemma \ref{sim1} we obtain that there are two cubes $c_1,c_2\in C^{k+1}(N)$ such that $c_1(0^{k+1})=x,c_2(0^{k+1})=z$ and $c_1(v)=c_2(v)=y$ for every $v\neq 0^{k+1}$. This means that $x\sim_k z$.
\end{proof}

\begin{lemma}\label{sim2} Two elements $x,y\in N$ satisfy $x\sim_k y$ if and only if for every cube $c_1\in C^{k+1}(N)$ with $c_1(0^{k+1})=x$ the map $c_2:\{0,1\}^{k+1}\rightarrow N$ satisfying
$$c_2(0^{k+1})=y~{\rm and}~c_2(v)=c_1(v)~~\forall~v\in\{0,1\}^{k+1}\setminus\{0^{k+1}\}$$ is in $C^{k+1}(N)$.
\end{lemma}

\begin{proof} Let $\phi=f^{k+1}:T_{k+1}\rightarrow\{0,1\}^{k+1}$ where $f(-1)=1,f(0)=0,f(1)=0$. Let $g:T_{k+1}\rightarrow N$ be the function obtained from $c_1\circ \phi$ by modifying the value on $1^{k+1}$ from $x$ to $y$.
Lemma \ref{sim1} guarantees that $g$ is a morphism. According to lemma \ref{extvert} the composition of $\omega$ and $g$ is in $C^{k+1}(N)$. On the other hand $c_2=g\circ\omega$.
\end{proof}

\begin{corollary}\label{sim2cor} For every $k\in\mathbb{N}$ and cube $c\in C^{k+1}(N)$ we have that if a function $c_2:\{0,1\}^{k+1}\rightarrow N$ satisfies $c(v)\sim_k c_2(v)$ for every $v\in\{0,1\}^{k+1}$ then $c_2\in C^{k+1}(N)$.
\end{corollary}

\begin{proof} We get the statement by iterating lemma \ref{sim2}. Note that by the symmetries of cubes the vector $0^{k+1}$ can be replaced by any other vector in lemma \ref{sim2}.
\end{proof}

\begin{corollary}\label{sim2cor2} A cube $c\in C^n(N/\sim_k)$ is uniquely determined by the elements $c(v)$ where $v\in\{0,1\}^n$ contains at most $k$ one's.
\end{corollary}

\begin{proof} For $n=k+1$ it follows directly from corollary \ref{sim2cor}. If $n>k+1$ then straightforward induction on the number of one's in $v$ complete the proof.
\end{proof}

\begin{lemma} For every $k\in\mathbb{N}$ and nilspace $N$ the equivalence relation $\sim_k$ is a congruence.
\end{lemma}

\begin{proof} Let $M=N/\sim_k$ with the induced cubespace structure. It is clear that $M$ satisfies the ergodicity property. We need to check the gluing axiom. Let $f:\{0,1\}^n\setminus\{1^n\}\rightarrow M$ be a map which is a morphism of the corner of the $n$ dimensional cube to the cubespace $M$.
We need to show that $f$ extends to the whole cube $\{0,1\}^n$ as a morphism.
Let $T$ be the subset in $\{0,1\}^n$ of vectors with at most $k+1$ one's in the coordinates. Corollary \ref{sim2cor} shows that the restriction of $f$ to $T$ can be lifted from $M$ to $N$ as a morphism. Let $\bar{f}$ denote a lift.
Lemma \ref{simpglue} implies that $\bar{f}$ extends to a morphism $\bar{f}_2$ of the whole cube $\{0,1\}^n$ to $N$.
It is easy to see that the composition (call it $f_2$) of $\bar{f}_2$ with the factor map $\pi:N\rightarrow M$ is equal to $f$ when restricted to $\{0,1\}^n\setminus\{1^n\}$. Now corollary \ref{sim2cor2} shows that the restriction of $f_2$ to each face in $\{0,1\}^n$ of dimension $n-1$ and containing $0^n$ is equal to $f$. This completes the proof.
\end{proof}

\begin{definition} For a nilspace $N$ we denote by $\mathcal{F}_k(N)$ the factor $N/\sim_k$. We say that $N$ is a {\bf $k$-step nilspace} if $N=\mathcal{F}_k(N)$.
\end{definition}

Another way of formulating the previous definition is that $N$ is a $k$-step nilspace if and only if every morphism of the corner of the $k+1$ cube to $N$ extends  in a unique way to a morphism of the $k+1$ dimensional cube.
In other words the gluing axiom for $k+1$ dimensional cubes holds in a stronger form where uniqueness of the extension is guaranteed.
Note that this {\bf unique closing property} also appears in the Host-Kra theory of parallelepiped structures.
The next simple lemma will be important.

\begin{lemma}\label{lifting} Assume that $P\subset\{0,1\}^n$ is a sub-cubespace in $\{0,1\}^n$ with the extension property, $k$ is some natural number and $N$ is a nilspace. Then every morphism $f:P\rightarrow\mathcal{F}_k(N)$ has a lift $f':P\rightarrow\mathcal{F}_{k+1}(N)$ such that $f'$ is a morphism and $f'\equiv f$ modulo $\sim_k$.
\end{lemma}

\begin{proof} The definition of the cubespace structure on $\mathcal{F}_k(N)$ immediately implies that the statement is true for $P=\{0,1\}^n$. If $P$ is smaller then we first extend $P$ to a morphism $g:\{0,1\}^n\rightarrow\mathcal{F}_k(N)$, then we lift $g$ to some morphism $g':\{0,1\}^n\rightarrow\mathcal{F}_{k+1}(N)$ and finally we restrict $g'$ to $P$. 
\end{proof}

\begin{lemma}\label{cubechar} Let $N$ be a $k$-step nilspace and $n\geq k+2$. A function $c:\{0,1\}^n\rightarrow N$ is in $C^n(N)$ if and only if its restrictions to $k+1$ dimensional faces with at least one point with $0$ in the last coordinate are all in $C^{k+1}(N)$.
\end{lemma}

\begin{proof} Let $P$ be the set of elements in $\{0,1\}^n$ with at most $k$ ones. Note that $P$ is the union of the $k$-dimensional faces containing $0^n$. The condition of the lemma implies that $c$ restricted to such faces are cubes. Using lemma \ref{simpglue} and the fact that $N$ is $k$-step we get that there is a unique element $c'$ in $C^n(N)$ whose restriction to $P$ is equal to the restriction of $c$ to $P$. We claim that $c=c'$. Let $t$ be the maximal integer such that $c=c'$ on every element $v\in\{0,1\}^n$ with at most $t$ ones in its coordinates. By contradiction assume that $t<n$. Then there is an element $w\in\{0,1\}^n$ with $t+1$ ones such that $c'(w)\neq c(w)$. Since $t>k$ It can be seen that $w$ is contained in a $k+1$ dimensional face $F$ such that every element in $F\setminus\{w\}$ has at most $t$ ones and furthermore there is at least one point if $F$ with $0$ in its last coordinates. Such a face can be found by choosing the last $k+1$ elements from the support of $w$ and then changing those coordinates in $w$.

We know that the restriction of $c$ to $F$ is in $C^{k+1}(N)$. Since there is only one way of completing $F\setminus\{w\}$ to a cube the proof is complete.
\end{proof}

\subsection{Linear and higher degree abelian groups}

We will see that abelian groups appear in the structures of nilspaces in various ways as building blocks.
Every abelian group $A$ has a natural nilspace structure that we call ``linear''.
Cubes in $C^n(A)$ are functions $f:\{0,1\}^n\rightarrow  A$ satisfying
\begin{equation}\label{linearcube}
f(e_1,e_2,\dots,e_n)=a_0+\sum_{i=1} e_ia_i
\end{equation}
for some elements $a_0,a_1,\dots,a_n\in A$.  There is however another
way of describing these functions. If $f$ satisfies (\ref{linearcube})
then for any cube morphism $\phi:\{0,1\}^2\rightarrow\{0,1\}^n$ we
have $$f(\phi(0,0))-f(\phi(0,1))-f(\phi(1,0))+f(\phi(1,1))=0$$ and it
is easy to see that it gives an alternative characterization for
linear cubes.  The advantage of the second description is that it can
be naturally generalized.  For an arbitrary map
$f:\{0,1\}^n\rightarrow A$ to an abelian group let us introduce the
weight of $f$ by
\begin{equation}\label{weight}
w(f)=\sum_{v\in\{0,1\}^{n}} f(v)(-1)^{h(v)}
\end{equation}
where $h(v)=\sum_{i=1}^{n} v_i$.

\begin{definition} For every $k\in\mathbb{N}$ and abelian group $A$ let us define the nilspace  $\mathcal{D}_k(A)$ on the point set $A$ in the following way.
A map $f:\{0,1\}^n\rightarrow A$ is in $C^n(\mathcal{D}_k(A))$ if and only if
for every morphism $\phi:\{0,1\}^{k+1}\rightarrow \{0,1\}^n$ we have that $w(f\circ \phi)=0$. We say that $\mathcal{D}_k(A)$ is the {\bf $k$-degree structure} on $A$.
\end{definition}

To check the gluing axiom in $\mathcal{D}_k(A)$ is a straightforward calculation.
Observe that $\mathcal{D}_k(A)$ is a $k$-step and $k$-fold ergodic nilspace. 
Lemma \ref{kabelian} will show the converse of this observation. The proof uses the next lemma which is interesting on its own right.

\begin{lemma}\label{onestep} One step nilspaces are affine abelian groups with the linear nilspace structure.
\end{lemma}

\begin{proof} Let $N$ be a one step nilspace. Let us distinguish an arbitrary element $e\in N$ and call it identity. For every $x,y\in N$ we define $x+y$ as the unique extension of the morphism defined by $f(0,0)=e,f(1,0)=x,f(0,1)=y$ (of the corner of the two dimensional cube) to $(1,1)$. We need to check the abalian group axioms.

Commutativity of $+$ follows directly from the symmetry of $\{0,1\}^2$ interchanging $(1,0)$ and $(0,1)$.

If $x,y,z\in N$ the we can extend the map $g(0,0,0)=e,g(1,0,0)=x,g(0,1,0)=y,g(0,0,1)=z$ to the full cube $\{0,1\}^3$. Let $g_2$ denote the extension. The composition of $g_2$ by the maps $\phi_1,\phi_2:\{0,1\}^2\rightarrow \{0,1\}^3~,~\phi_1(a,b)=(a,a,b)$ and $\phi_2(a,b)=(a,b,b)$ shows associativity.

If $f(0,0)=x,f(1,0)=e,f(0,1)=e$ then the unique extension $y=f(1,1)$ satisfies $x+y=e$.
\end{proof}

\begin{lemma}\label{kabelian} If $N$ is a $k$-step, $k$-fold ergodic nilspace then $N$ is isomorphic to $\mathcal{D}_k(A)$ for some abelian group $A$.
\end{lemma}

\begin{proof} We use induction on $k$. Lemma \ref{onestep} shows the statement for $k=1$. Assume that $k\geq 2$ and the statement is already proved for $k-1$. Let $e$ be a fixed element in $N$. After $k-1$ repeated applications of $\delta_e$ to $N$ we obtain a $1$-step nilspace $\partial_e^{k-1}N$.  The condition that $N$ is a single class of $\sim_{k-1}$ implies by lemma \ref{sim2cor} that every function $f:\{0,1\}^k\rightarrow N$ is a cube. In particular $\partial_e^{k-1}N$ is ergodic. Lemma \ref{onestep} implies that $\partial_e^{k-1}N$ is isomorphic to an abelian group $A$ with the linear structure.

Let $M$ be the arrow space over $N$. Since $k\geq 2$ we have that $M$ is ergodic. Cubes of dimension $k+1$ in $N$ are in a one to one correspondence with cubes of dimension $k$ in $M$.
We claim that two arrows $x=(x_1,x_2)$ and $y=(y_1,y_2)$ in $M$ are $\sim_{k-1}$ equivalent if and only if $x_1-x_2=y_1-y_2$ in $A$. First notice that $M$ is in a single $\sim_{k-2}$ class and so the factor $\mathcal{F}_{k-1}(M)$ satisfies the condition of the lemma with $k-1$. Let $f:\{0,1\}^k\rightarrow M$ be the map defined in a way that $f(0^k)=x$, $f(1,0,0,\dots,0)=y$ and $f(v)=e$ everywhere else. The induction hypothesis guarantees that $x=y$ in the factor $\mathcal{F}_{k-1}(M)$ if and only if $f$ is a cube in $M$. This shows that $x\sim_{k-1} y$ if and only if $x_1,x_2,y_1,y_2$ form a two dimensional cube in $\partial_e^{k-1}N=A$. This proves the claim.

We obtain from the claim that if $c\in C^{k+1}(N)$ is an arbitrary cube then if we add the same element in $a\in A$ to the $c$ values of two endpoints of an arbitrary edge in $\{0,1\}^{k+1}$ then the resulting new function is still a cube.
By repeating this operation we can produce a new cube $c'$ in which all but one of the vertices are mapped to $e$. Using that constant functions are all cubes and the unique closing property we obtain that $c'$ has to be the constant function.
In other words $c$ can be obtained from the constant function with the inverses of the previous operations which shows that all the cubes are in $\mathcal{D}_k(A)$. The fact that every $2^{k+1}-1$ points can be completed to a cube shows that the cubes in $N$ are exactly the cubes in $\mathcal{D}_k(A)$.

\end{proof}

\begin{corollary} If $N$ is a $k$-step nilspace then every equivalence class of $\sim_{k-1}$ is an abelian group with the $k$-degree structure.
\end{corollary}

\subsection{Bundle decomposition of nilspaces}\label{bundledecomp}

We give a structure theorem for $k$-step nilspaces which follows relatively easily from the axioms but which is useful as an intermediate step to prove deeper structure theorems.

\begin{definition} Let $A$ be an abelian group. An (abstract) {\bf $A$-bundle} over a set $S$ is a set $T$ with an action $\alpha:A\times T\rightarrow T$ and a bundle map $\pi:T\rightarrow S$ such that
\begin{enumerate}
\item the action $\alpha$ is free i.e. the stabilizer of every element is the trivial subgroup in $A$,
\item $\pi$ gives a bijection between the orbits of $A$ in $T$ and the elements of $S$.
\end{enumerate}
A {\bf $k$-fold abelian bundle} with structure groups $A_1,A_2,\dots, A_k$ is the last member of a sequence $T_0,T_1,\dots,T_k$ of ``factors'' where $T_0$ is a one element set and $T_i$ is an $A_i$ bundle over $T_{i-1}$.
These objects come together with projections (bundle maps) $\pi_{i,j}:T_i\rightarrow T_j$ for $i\geq j$.
By abusing the notation we use the shorthand notation $\pi_j$ for $\pi_{i,j}$.
A {\bf relative $k$-fold abelian bundle} is a generalization of a $k$-fold abelian bundle whose base $T_0$ can be arbitrary. 
\end{definition}

Note that if $T$ is an $A$-bundle over $S$ then fibers (preimages of points under $\pi$) can be regarded as affine versions of $A$. We will use the short hand notation $x+a$ for $\alpha(a,x)$. There is no distinguished bijection between the elements of a fiber $F$ and $A$ but there is a well defined difference map $F\times F\rightarrow A$ which, if $x,y\in F$, is given by the unique element in $a\in A$ satisfying $y+a=x$. We simply denote the difference of $x$ and $y$ by $x-y$.

\begin{definition} A {\bf degree-$k$ bundle} $N$ is a cubespace that also has the structure of a $k$-fold abelian bundle with structure groups $A_1,A_2,\dots,A_k$ and factors\\ $T_0,T_1,\dots,T_k=N$ with the following property: For every $0\leq i\leq k-1$, $n\in\mathbb{N}$ and $c\in C^n(T_{i+1})$ we have that
$$\{c_2|c_2\in C^n(T_{i+1})~,~\pi_i\circ c=\pi_i\circ c_2\}=\{c+c_3| c_3\in C^n(\mathcal{D}_{i+1}(A_{i+1}))\}$$
where $C^n(T_i)=\pi_i(C^n(N))$.
\end{definition}

\begin{theorem}[Bundle decomposition]\label{bundec} A cubespace $N$ is a degree-$k$ bundle if and only if $N$ is a $k$-step nilspace. Furthermore $\mathcal{F}_i(N)$ is equal to $T_i$ for every $1\leq i\leq k$.
\end{theorem}

\begin{proof} First we show that if $N$ is a degree-$k$ bundle then it is a $k$-step nil-space. It is clear that $N$ satisfies the ergodicity axiom. It remains to show the gluing axiom. We use induction on $i$ to prove it in $T_i$. If $i=0$ then the statement is trivial.

 Assume that we have gluing in $T_i$. Let
 $f:\{0,1\}^n\setminus\{1^n\}\rightarrow T_{i+1}$ be a morphism of the
 corner of the $n$-dimensional cube. The map $\pi_i \circ f$ has an extension $f_2:\{0,1\}^n\to T_i$ to the full cube. Since $C^n(T_i)=\pi_i(C^n(N))$ we have that $f_2$ can be lifted (with respect to $\pi_i$) to a morphism $f_3:\{0,1\}^n\to T_{i+1}$. Let us consider $f_4=f-f_3$ on $\{0,1\}^n\setminus\{1^n\}$. It follows by definition that $f_4$ is a morphism of the corner to $\mathcal{D}_{i+1}(A_{i+1})$ and so it can be extended to a morphism $f_5:\{0,1\}^n\to \mathcal{D}_{i+1}(A_{i+1})$. Now it is clear that $f_3+f_5$ is an extension of $f$ to the full cube. The definition of degree-$i$ bundles implies that $\mathcal{F}_i(N)=T_i$.

\bigskip

We prove the other direction by induction on $k$. The step $k=0$ is trivial. Assume that it holds for $k-1$ and $N$ is a $k$-step nilspace. By induction we have the $k-1$ degree bundle structure on $\mathcal{F}_{k-1}(N)$.

First we show that every $\sim_{k-1}$ class $F$ (with the restricted cubic structure) is isomorphic to $\mathcal{D}_k(A_F)$ for some abelian group $A_F$. 
By lemma \ref{kabelian} it is enough to check that $F$ is a $k$-fold ergodic, $k$-step nilspace.
If $x\in F$ is a arbitrary element then the constant $x$ function on $\{0,1\}^k$ is in $C^k(N)$ and so by lemma \ref{sim2cor} every function $\{0,1\}^k\rightarrow F$ is in $C^k(N)$. Let $f:\{0,1\}^n\setminus\{1^n\}\rightarrow F$ be a corner where $n\geq k+1$. Since $N$ is a $k$ step nilspace we have that $f$ has a unique completion to a cube $f':\{0,1\}^n\rightarrow N$.  In the factor $\mathcal{F}_{k-1}(N)$ the function $f$ becomes constant and so the only completion is the constant function. This shows that $f'$ maps $\{0,1\}^n$ to $F$ and thus $F$ is a $k$-step nilspace.

Let $M=\{(x,y)|x,y\in N,~x\sim_{k-1} y\}\subset N\times N$. Note that $F\times F\subset M$ holds for every class $F$ of $\sim_{k-1}$.
We introduce an equivalence relation $\sim$ on $M$. Let $F_1,F_2$ be two $\sim_{k-1}$ classes of a $k$-step nilspace $N$. If $x_1,x_2\in F_1$ and $y_1,y_2\in F_2$ then we say that $(x_1,x_2)\sim(y_1,y_2)$ if $(x_1,y_1)\sim_{k-1} (x_2,y_2)$ in the arrow space $N'$ of $N$. Note that $N'$ is not necessarily ergodic but it will not cause any problem.

A simple description of the equivalence relation $\sim$ follows from lemma \ref{sim1}. Let $c:\{0,1\}^{k+1}\rightarrow N$ be the function such that $c(v,0)=x_1,c(v,1)=y_1$ if $v\in\{0,1\}^k\setminus\{1^k\}$ and $c(1^k,0)=x_2,c(1^k,1)=y_2$. Then $(x_1,x_2)\sim(y_1,y_2)$ if and only if $c\in C^{k+1}(N)$. In particular it follows that for every $x_1,x_2\in F_1$ and $y_1\in F_2$ there is a unique $y_2$ such that $(x_1,x_2)\sim(y_1,y_2)$. This follows from the fact that $c$ restricted to $\{0,1\}^{k+1}\setminus\{1^{k+1}\}$ is a corner and since $N$ is a $k$-step nilspace it has a unique completion $y_2$. The fact that $y_2$ has to be in $F_2$ follows by taking the situation modulo $\sim_{k-1}$ where the unique completion has to be congruent to $y_1$ since $x_2$ is congruent to $x_1$. 

If $F$ is a $\sim_{k-1}$ class and $x_1,x_2,y_1,y_2\in F$ then $(x_1,x_2)\sim(y_1,y_2)$ if and only if $x_2-x_1=y_2-y_1$ holds in the abelian groups $A_F$. In other words, inside one class of $\sim_{k-1}$, the elements of $A_F$ are in a bijection with the $\sim$ classes of vectors in such a way that $a\in A_F$ corresponds to the class of pairs of the form $(x,x+a)$. 
Using this, if $F_1$ and $F_2$ are two $\sim_{k-1}$ classes then there is a natural bijection $\phi$ between $A_{F_1}$ and $A_{F_2}$ such that $\phi(a)=b$ if and only if $(x,x+a)\sim (y,y+b)$ for every $x\in F_1$ and $y\in F_2$.

We show that the map $\phi$ is an isomorphism between $A_{F_1}$ and $A_{F_2}$.
It is clear from the definition of $\sim$ that if $(x_1,x_2)\sim (y_1,y_2)$ and $(x_2,x_3)\sim(y_2,y_3)$ then $(x_1,x_3)\sim(y_1,y_3)$. Inside one fiber the class of $(x_1,x_3)$ corresponds to the sum of the group elements represented by $(x_1,x_2)$ and $(x_2,x_3)$. It follows that  $\phi$ preserves addition in both directions and so it is a group isomorphism.

Let us denote by $A$ the unique abelian group formed by the $\sim$ classes in $F\times F$ for each $\sim_{k-1}$ class $F$. The group $A$ acts on each $\sim_{k-1}$ class and so on the whole space $N$. We denote this action by simple addition. This action satisfies that if $x\in F_1,y\in F_2$ the $(x,x+a)\sim (y,y+a)$ for every $a\in A$. It follows that if $c:\{0,1\}^{k+1}\rightarrow N$ is any cube and $a\in A$ then by applying the action of $a$ to the two endpoint of an arbitrary edge in $c$ we get a cube.
Assume now that two cubes $c_1$ and $c_2$ in $C^{k+1}(N)$ satisfy that $c_1\sim_{k-1} c_2$. Then by repeating the previous operations we can create a new cube $c_2'$ from $c_2$ that differs from $c_1$ at most at one vertex. Using the unique closing property this implies that $c_2'=c_1$ and $c_2-c_1\in\mathcal{D}_k(A)$.
\end{proof}

An interesting consequence of theorem \ref{bundec} is that in a $k$-step nilspace $N$ the $\sim_{k-1}$ classes are all isomorphic abelian groups with $k$-degree structures and there is a distinguished set of affine isomorphisms between any two of them. Let $F_1$ and $F_2$ be $\sim_{k-1}$ classes and let us fix elements $x\in F_1$ and $y\in F_2$. Then the map $\phi(x+a)=y+a~,~a\in A_k$ defines an affine morphism between $F_1$ and $F_2$.
Such maps will be called {\bf local translations}.
The next characterization of local translations follows directly from theorem \ref{bundec}.

\begin{lemma}\label{loctrans} Let $N$ be a $k$-step nilspace. Let us fix two $\sim_{k-1}$ classes $F_1,F_2$ and two elements $x\in F_1,y\in F_2$. For every $z\in F_1$ we denote by $\phi_{x,y}(z)$ the unique closure of the corner $c:\{0,1\}^{k+1}\setminus\{1^{k+1}\}\rightarrow N$ defined by $c(v,0)=x$ if $v\neq (1^k,0)$, $c(1^k,0)=z$ and $c(v,1)=y$ if $v\in\{0,1\}^k\setminus\{1^k\}$. Then the map $\phi_{x,y}$ is the local translation corresponding to $x$ and $y$.
\end{lemma}

\subsection{Sub-bundles and bundle morphisms}\label{subbundchap}

This is a very technical part of the paper. The main application is that the results below will help us in putting a probability space structure on homomorphism sets. 

\begin{definition} Let $T_k$ be a $k$-fold abelian bundle with structure groups $A_1,A_2,\dots, A_k$, factors $T_0,T_1,\dots,T_k$ and projections $\pi_1,\pi_2,\dots,\pi_k$.
We define the notion of a {\bf sub-bundle} of $T_k$ with structure groups $A_1'\leq A_1,A_2'\leq A_2,\dots,A_k'\leq A_k$ and factors $T'_0=T_0,T'_1\subseteq T_1,\dots,T'_k\subseteq T_k$.
If $k=0$ then $T'_0=T_0$ and both are equal to a one point space.
For a general $k$ we have the condition that $T'_{k-1}=\pi_{k-1}(T'_k)$ is already a sub-bundle and for every $x\in T_k'$ we have that $$\{a~|~a\in A_k,a+x\in T_k'\}=A_k'.$$
In particular if $k=1$ then a sub-bundle is just a coset of $A_1'$.
\end{definition}

An important example for sub-bundles is the following.
Let $P=\{0,1\}^n$ be a cube and $N$ be a $k$-step nilspace. Let us consider the natural embedding $\Hom(P,N)$ into the direct power $N^P$. This means that every homomorphism $\phi:P\rightarrow N$ is represented by the vector whose component at coordinate $p\in P$ is $\phi(p)$.
According to theorem \ref{bundec}, $\Hom(P,N)$ is a sub-bundle in $N^P$ with structure groups $\Hom(P,\mathcal{D}_i(A_i))$.

\begin{definition}\label{bundmorphdef} Let $T=T_k$ and $T'=T'_k$ be two $k$-fold abelian bundles with structure groups $\{A_i\}_{i=1}^k$,$\{A_i'\}_{i=1}^k$ and factors $\{T_i\}_{i=0}^k$,$\{T'_i\}_{i=0}^k$. We define the notion of a {\bf bundle morphism} $\psi:T\rightarrow T'$ with structure morphisms $\{\alpha_i:A_i\rightarrow A_i'\}$ by the next two axioms.
\begin{enumerate}
\item If $1\leq i\leq k$ we have  $\pi_i(x)=\pi_i(y)$ then $\pi_i(\psi(x))=\pi_i(\psi(y))$. In other words $\psi$ induces well defined maps $\psi_i:T_i\rightarrow T'_i$
\item $\psi_i(x+a)=\psi_i(x)+\alpha_i(a)$ where $1\leq i\leq k$,~$x\in T_i$~and $a\in A_i$.
\end{enumerate}
We say that $\psi$ is {\bf totally surjective} if all the structure morphisms are surjective.
\end{definition}

Now we generalize the concept of the kernel of a homomorphism between abelian groups to totally surjective bundle morphisms. The generalized kernel is a relative abelian bundle.

\begin{definition} Let $\psi:T\to T'$ be a totally surjective bundle morphism. We use the notation of definition \ref{bundmorphdef}.  The {\bf kernel} of $\psi$ is a relative $k$-fold abelian bundle $K$ on $T$ with base $T'$ and structure groups $\{\ker(\alpha_i)\}_{i=1}^k$ defined in the following way. Let $$K_i=\{(x,y)\in T_i\times T'_k|\psi_i(x)=\pi_i(y)\}.$$ For $1\leq i\leq k$ we let $\ker(\alpha_i)$ act on $K_i$ by $(x,y)+a=(x+a,y)$. Furthermore the projection $\pi_i:K_j\to K_i$ is defined by $\pi_i(x,y)=(\pi_i(x),y)$ for $j\geq i$. 

\end{definition}

\begin{remark}\label{kerident} We will identify $K_k$ with $T_k$ using the bijection $(x,y)\leftrightarrow x$ and $K_0$ with $T'_k$ using the bijection and $(x,y)\leftrightarrow y$. It is easy to see that under these identifications, $\pi_0 : K_k \to K_0$ becomes identified with $\psi : T \to T'$. 
\end{remark}

The next lemma justifies the previous definition.

\begin{lemma} The map $\pi_{i-1}:K_i\to K_{i-1}$ is a $\ker(\alpha_i)$ bundle.
\end{lemma}

\begin{proof} First we check that $\pi_{i-1}$ is surjective. Let $(x,y)\in K_{i-1}$ which means that $\psi_{i-1}(x)=\pi_{i-1}(y)$.  Let $z\in T_i$ be such that $\pi_{i-1}(z)=x$.  Since $\pi_{i-1}(\psi_i(z))=\psi_{i-1}(\pi_{i-1}(z))=\psi_{i-1}(x)=\pi_{i-1}(y)$ there exists $a'\in A'_i$ such that $\psi_i(z)+a'=\pi_i(y)$. Since $\alpha_i$ is surjective there is $a\in A_i$ with $\alpha_i(a)=a'$. Then the pair $(z+a,y)$ is in $K_i$ and maps to $(x,y)$.   

Clearly $\ker(\alpha_i) \subset A_i$ acts freely on $K_i$, so it remains to check that it acts transitively on the fibers of $\pi_{i-1}$. Let $(x_1,y)$ and $(x_2,y)$ with  be any elements in the same fiber of $\pi_{i-1}:K_i\to K_{i-1}$. Since $\pi_{i-1}(x_1) = \pi_{i-1}(x_2)$ there is $a\in A_i$ with $x_1=x_2+a$. Then $\pi_i(y)=\psi_i(x_1)=\psi_i(x_2+a)=\psi_i(x_2)+\alpha_i(a)=\pi_i(y)+\alpha_i(a)$ which implies that $a\in\ker(\alpha_i)$. 
\end{proof}

\begin{lemma}\label{preimbund} Let $\psi:T\to T'$ be a totally surjective bundle morphism. Then for every $t'\in T'$ we have that $\psi^{-1}(t')$ is a sub-bundle of $T$ with structure groups $\{\ker(\alpha_i)\}_{i=1}^k$.
\end{lemma}

\begin{proof} This can be easily seen by using the relative abelian bundle $K = \ker(\psi)$.
First, it is clear that the fiber of $K$ above $t'$ inherits the structure of a $k$-fold abelian bundle with structure groups $\{\ker(\alpha_i)\}_{i=1}^k$. By remark \ref{kerident}, this fiber is $\psi^{ -1}(t')$. Next, it is easy to see that $K$
itself is ``like'' a sub-bundle of $T$, more precisely, it satisfies the conditions in the definition
of sub-bundle except that it is not true that $K_i \subseteq T_i$. Instead we have the maps
$K_i \to T_i$ defined by $(x,y) \mapsto x$. These maps are not inclusions, of course,
but they become inclusions when restricted to $\psi^{-1}(t')$, so that $\psi^{-1}(t')$ is
a sub-bundle of $T$.
\end{proof}

\begin{lemma}\label{morphbund} A morphism $\psi$ between two $k$-step nilspaces $N$ and $N'$ induces a bundle morphism between the corresponding $k$-degree bundles $T$ and $T'$.
\end{lemma}

\begin{proof} Lemma \ref{sim1} shows that if $x\sim_i y$ then $\psi(x)\sim_i\psi(y)$. This verifies the first axiom.

First we prove the second axiom when the nil-spaces are of the form $\mathcal{D}_i(A_i)$ and $\mathcal{D}_i(A_i')$. The abelian group structure of $A_i$ and $A_i'$ can be recovered by applying $\partial_x^{i-1}$ to the cubic structure with some fixed element $x$ in $A_i$ or $A_i'$. It is clear that $\psi_i$ preserves this structure and so $\psi_i$ has to be an affine homomorphism between the two abelian groups which means that $\psi_i(x+a)=\psi_i(x)+\alpha(a)$ where $\alpha$ is a homomorphism.

Now let $F$ be a $\sim_{i-1}$ class in $T_i$. Then $F=\mathcal{D}(A_i)$ and by the first part of the proof we have that $\psi_i$ restricted to $F$ satisfies $\psi_i(x+a)=\psi_i(x)+\alpha_F(a)$ where $x\in F,a\in A_k$ and $\alpha_F:A_i\rightarrow A'_i$ is a group homomorphism.

It remains to show that we have the same group homomorphism $\alpha_F$ corresponding to each $\sim_{i-1}$ class.
This follows from the fact that the relation $\sim$ defined in the proof of $\ref{bundec}$ is preserved under $\psi_i$ because it is defined through cubes. 
\end{proof}

\subsection{Fiber surjective morphisms}

In this chapter we introduce a family of morphisms which have useful properties.
Such morphisms come up naturally in various structure theorems (see for example theorem \ref{invlim}).

\begin{definition} A morphism $\psi: N_1\rightarrow N_2$ between two nilspaces will be called {\bf fiber surjective} if for every $n\in\mathbb{N}$ the image of a $\sim_n$ class in $N_1$ is a $\sim_n$ class in $N_2$.
\end{definition}

This definition immediately implies that if $N_1$ is a $k$-step nilspace then $N_2$ is an at most $k$-step nilspace.
It is also clear that $\psi$ induces fiber surjective maps between the factors $\mathcal{F}_i(N_1)$ and $\mathcal{F}_i(N_2)$ for every natural number $i$.
Note that lemma \ref{morphbund} implies that a fiber surjective map between two $k$-step nilspaces is a totally surjective bundle morphism between the corresponding $k$-fold bundles.
A natural way of viewing a fiber surjective map $\psi:N_1\rightarrow N_2$ is that $N_2$ is a factor space of $N_1$ using the equivalence relation $x\sim y~\Longleftrightarrow~\psi(x)=\psi(y)$.
The next lemma verifies that the induced cubic structure on $N_1/\sim$ is identical with the cubic structure on $N_2$ and thus $\psi$ is a factor map.

\begin{lemma}\label{lifting2} Let $\phi: N\rightarrow N'$ be a fiber
  surjective morphism between two $k$-step nilspaces. Then every cube
  $c\in C^n(N')$ can be lifted to a cube $c'\in C^n(N)$ such that
  $\phi \circ c'=c$. In other words $N'$ is a factor nilspace of $N$.
\end{lemma}

\begin{proof} The proof is an induction on $k$. If $k=0$ then there is nothing to prove.
  Assume that we have the statement for $k-1$. The map $\phi$ induces
  a map $\phi'$ from $\mathcal{F}_{k-1}(N)$ to
  $\mathcal{F}_{k-1}(N')$. This means (using lemma \ref{lifting}) that
  there is a cube $c_2\in C^n(N)$ such that $\phi\circ c_2\sim_{k-1} c$
  and so $c_3=\phi\circ c_2-c$ is in $C^n(\mathcal{D}(A_k'))$. Now it is enough to find a lift $c_4$ of $c_3$ under the surjective homomorphism $\alpha_k:A_k\rightarrow A_k'$ because then $c_2-c_4$ is a lift of $c$.
The existence of $c_4$ follows by first considering an arbitrary lift of a $k$-dimensional corner of $c_3$ and then by extending it (uniquely) to an $n$-dimensional cube.
\end{proof}

An important example of a fiber surjective map is the following.
Let $N$ be a $k$-step nilspace with structure groups $A_1,A_2,\dots,A_k$ and let $B\subseteq A_k$ be a subgroup of $A$. We introduce a nilspace denoted by $N/B$ in the following way. Let us say that two elements $x,y\in N$ satisfy $x\sim_B y$ if $x\sim_{k-1} y$ and $x-y\in B$. The elements of $N/B$ are the equivalence classes of $\sim_B$. It follows from theorem \ref{bundec} that $N/B$ is a factor of $N$ and the projection $N\rightarrow N/B$ is fiber surjective.

\subsection{Restricted morphisms}

\begin{definition}Let $X\subset P$ be a subset of the cubespace $P$ and let $f:X\rightarrow N$ be an arbitrary function.
We define the {\bf restricted homomorphism set} $\Hom_f(P,N)$ as the collection of those morphisms whose restrictions to $X$ is equal to $f$.
\end{definition}

\begin{lemma}\label{restbund} If $X$ is a sub-cubespace of $P$ with the extension property in $P$ and $P$ is a sub-cubespace of $\{0,1\}^n$ with the extension property then for every morphism $f:X\rightarrow N$ into some finite step nilspace $N$ the restricted homomorphism set $\Hom_f(P,N)$ is a sub-bundle in $N^P$ with structure groups $$\Hom_{X\rightarrow 0}(P,\mathcal{D}_i(A_i))$$ where $A_i$ is the $i$-th structure group of $N$.
\end{lemma}

\begin{proof} To obtain this bundle structure we use an iterative argument.
Assume that $N$ is a $k$-step nilspace and the result is already established for $N_{k-1}=\mathcal{F}_{k-1}(N)$.
Let $f_2:P\rightarrow N_{k-1}$ be a morphism whose restriction to $X$ is $f$ modulo $\sim_{k-1}$. We claim that $f_2$ can be lifted to an element of $f_3\in\Hom_f(P,N)$ and that the set of possible lifts is exactly $f_3+\Hom_{X\rightarrow 0}(P,\mathcal{D}_k(A_k))$. 

Lemma \ref{lifting} implies that there is a lift $g:P\rightarrow N$ of $f_2$.
Then the function $g_2=g|_X-f$ is a morphism from $X$ to $\mathcal{D}_k(A_k)$. By the extension property there is a morphism $g_3:P\rightarrow\mathcal{D}_k(A_k)$ extending this morphism. Then $f_3=g-g_3$ is in $\Hom_f(P,N)$ and is a lift of $f_2$. 
The second claim is trivial.
\end{proof} 

\begin{remark}\label{restrem} The proof of lemma \ref{restbund} does not require the full extension property. It is enough to assume that morphisms into finite step nilspaces have extensions. This will be important in a special construction later. 
\end{remark}

\begin{lemma}\label{collection} Let $P\subseteq\{0,1\}^n$ be a sub-cubespace with the extension property in $\{0,1\}^n$ and $X\subset P$ be a sub-cubespace with the extension property in $P$. Let $\psi:N\rightarrow N'$ be a fiber surjective morphism between two $k$-step nilspaces.
Then
\begin{enumerate}
\item $\Hom(P,N)$ is a sub-bundle in the direct power $N^P$ with structure groups $\Hom(P,\mathcal{D}_i(A_i))$
\item $\psi^P:\Hom(P,N)\rightarrow\Hom(P,N')$ is a totally surjective bundle morphism with structure morphisms $$\alpha_i^P:\Hom(P,\mathcal{D}_i(A_i))\rightarrow\Hom(P,\mathcal{D}_i(A_i'))$$
\item The preimage of $t\in\Hom(P,N')$ under $(\psi^P)^{-1}$ is a bundle with structure groups $\Hom(P,\mathcal{D}_i({\rm ker}(\alpha_i)))$.
\item Let $t\in\Hom(P,N')$ and let $t_2\in\Hom(X,N')$ be its restriction to $X$. Then the projection $\pi_{X}$ from $(\psi^{P})^{-1}(t)$ to $(\psi^{P_2})^{-1}(t_2)$ is a totally surjective bundle morphism.
\end{enumerate}
\end{lemma}

\begin{proof}

We prove the first statement by induction on $k$. For $k=0$ it is trivial. If we know the statement for $k-1$ then we have by lemma \ref{lifting} $\Hom(P,\mathcal{F}_{k-1}(N))=\pi_{k-1}(\Hom(P,N))$ and so we have that $\pi_{k-1}(\Hom(P,N))$ is a sub-bundle of $\mathcal{F}_{k-1}(N)^P$. Let $g\in\Hom(P,\mathcal{F}_{k-1}(N))$. If $g'$ is any lift of $g$ to $N$ then by theorem \ref{bundec} the other preimages of $g$ are exactly those that differ from $g'$ by an element in $\Hom(P,\mathcal{D}_k(A_k))$, which is clearly a subgroup in $A_k^P$.

For the second statement we check the two axioms of bundle morphisms. The first axiom follows from the fact (use lemma \ref{sim1}) that the map $\psi^P$ preserves the relation $\sim_i$. Let $c\in\Hom(P,\mathcal{F}_i(N))$. It is clear that the structure morphisms are given by $\alpha_i^P$ on $\Hom(P,\mathcal{D}_i(A_i))$ but we have to show that they map surjectively to $\Hom(P,\mathcal{D}_i(A_i'))$. This follows from lemma \ref{lifting2}.

The third statement follows directly from lemma \ref{preimbund}.

In the fourth statement the structural maps are computed as  $$\Hom(P,\mathcal{D}_i({\rm ker}(\alpha_i)))\rightarrow\Hom(X,\mathcal{D}({\rm ker}(\alpha_i))).$$ It follows from the extension property of $X$ that these are surjective maps.
\end{proof}

\subsection{Extensions and cohomology}

\begin{definition}\label{kdegext} Let $N$ be an arbitrary nilspace. A degree $k$-extension of $N$ is an abelian bundle $M$ over $N$ which is a cube space with the following properties.
\begin{enumerate}
\item For every $n\in\mathbb{N}$ we have $\pi(C^n(M))=C^n(N)$,
\item If $c_1\in C^n(M)$ and $c_2:\{0,1\}^n\rightarrow M$ with $\pi(c_1)=\pi(c_2)$ then $c_2\in C^n(M)$ if and only if $c_1-c_2\in C^n(\mathcal{D}_k(A))$.
\end{enumerate}
The map $\pi$ is the projection from $M$ to $N$.
The extension $M$ is called a split extension if there is a cube
preserving morphism $m:N\rightarrow M$ such that $\pi\circ m$ is the identity map of $N$.
\end{definition}

A motivation to study such extensions is that we can obtain every $k$-step nilspace from a trivial nilspace by $k$ consecutive extensions of increasing degree.
To understand extensions we will need the notion of cocycles.

If $\sigma$ is an automorphism of the cube $\{0,1\}^k$ then the we define $s(\sigma):=(-1)^m$ where $m$ is the number of $1$'s in the vector $\sigma(0^k)$. The automorphism $\sigma$ also acts on $C^k(N)$ by composition.

\begin{definition}\label{cocycle} Let $N$ be a nilspace and $A$ be an abelian group. A cocycle of degree $k-1$ is a function $\rho:C^k(N)\rightarrow A$ with the following two properties. 
\begin{enumerate}
\item If $f\in C^k(N)$ and $\sigma\in{\rm aut}(\{0,1\}^k)$ then $\rho(\sigma(f))=s(\sigma)\rho(f)$.
\item If $f_3$ is the concatenation of two cubes $f_1,f_2\in C^k(N)$ then $\rho(f_3)=\rho(f_1)+\rho(f_2)$. 
\end{enumerate}
Let $Y_k(N,A)$ denote the abelian group of all $A$-valued cocycles of degree $k$ with respect to the pointwise addition. In particular $Y_{-1}(N,A)$ denotes the set of all $A$-valued functions on $N$.
\end{definition}

Let $t\in\Hom(T_k,N)$, $G$ be an abelian group and $\tau:C^k(N)\rightarrow G$ be an arbitrary function.
We introduce $\beta(t,\tau)\in G$ by
\begin{equation}\label{altcubesum2}
\beta(t,\tau)=\sum_{v\in\{0,1\}^k}\tau(t \circ \Psi_v)(-1)^{h(v)}
\end{equation}
where $h(v)=\sum_{i=1}^k v_i$.

Since the outer cube $\{-1,1\}^n$ of $T_n$ can be obtained iteratively by concatenating cubes of the form $\Psi_v$ the next lemma is easy to see. 

\begin{lemma}\label{threecsum} Let $T_n$ be the three-cube of dimension $n$ and let $t:T_n\rightarrow N$ be a morphism into a nilspace $N$.
Assume that $\rho$ is a cocycle on $N$ of degree $n-1$. Then 
$$\beta(t,\rho)=\rho(t \circ \omega).$$
\end{lemma}

Any $k-1$ degree cocycle $\rho$ induces a degree $k$ cocycle that we denote by $\partial\rho$.
The value of $\partial\rho$ on $c\in C^{k+1}(N)$ is the difference of $\rho$ on two opposite faces of $c$.
To be more precise let $c_0$ and $c_1$ be the restrictions of $c$ to the cubes $\{0,1\}^k\times\{0\}$ and $\{0,1\}^k\times\{1\}$. Then $\partial\rho(c)=\rho(c_0)-\rho(c_1)$.
We have that $\partial:Y_{k-1}(N,A)\rightarrow Y_k(N,A)$ is a homomorphism for every $k\geq 0$. 

\begin{definition} A coboundary of degree $k$ is an element of the group $\partial^{k+1}Y_{-1}(N,A)$. 
We denote by $H_k(N,A)$ the factor group $Y_k(N,A)/\partial^{k+1}Y_{-1}(N,A)$. 
\end{definition}
We show that the elements of $Y_k(N,A)$ describe degree $k$-extensions of $N$. Two extensions are equivalent if they have the same image in the factor $H_k(N,A)$. In particular coboundaries represent split extensions.

First we show that extensions of degree $k$ generate cocycles of degree $k$.
The other direction will be shown in the chapter of measurable cocyles.

Let $M$ be a degree $k$-extension of $N$ by the abelian group $A$. Let $\pi$ denote the projection $M\rightarrow N$.
For every $n\in N$ we choose an element $x(n)$ with $\pi(x(n))=n$. Such a function $x$ will be called a {\bf cross section}. Furthermore if $m\in M$ is an arbitrary element then we introduce $f(m)=m-x(\pi(m))$.
For an arbitrary cube $c\in C^{k+1}(M)$ we define 
$$\rho(c)=\sum_{v\in\{0,1\}^{k+1}}f(c(v))(-1)^{h(v)}.$$
It is clear that the value of $\rho$ depends only on the $\pi$ image of $c$ and so we can also interpret $\rho$ as a function $\rho:C^{k+1}(N)\rightarrow A$.
It is also clear that $\rho$ satisfies the cocycle axioms. We say that $\rho$ is the cocycle generated by the cross section $x$.

\subsection{Translations}

For an arbitrary subset $F$ in $\{0,1\}^n$ and map $\alpha:N\rightarrow N$ we define the map $\alpha^F$ from $C^n(N)$ to $N^{\{0,1\}^n}$ such that $\alpha^F(c)(v)=\alpha(c(v))$ if $v\in F$ and $\alpha^F(c)(v)=c(v)$ if $v\notin F$.

\begin{definition}\label{transdef} Let $N$ be a nilspace. A map $\alpha:N\rightarrow N$ is called a translation of height $i$
if for every natural number $n\geq i$, $n-i$ dimensional face $F\subseteq\{0,1\}^n$ and $c\in C^n(N)$ the map $\alpha^F(c)$ is in $C^n(N)$. We denote the set of height $i$ translations by $\aff_i(N)$. We will use the short hand notation $\aff(N)$ for $\aff_1(N)$.
\end{definition}

Note that transformations in $\aff_i(N)$ were first introduce by Host
and Kra \cite{HKr}, \cite{HKr2} related to both ergodic theory and parallelepiped structures. 
It is clear from this definition that $$\aff_1(N)\supseteq\aff_2(N)\supseteq\aff_3(N)\supseteq\dots.$$

\begin{lemma}\label{transmorph} A map $\alpha:N\rightarrow N$ is in $\aff_i(N)$ if and only if the map $h:N\rightarrow N\times N$ defined by $h(n)=(n,\alpha(n))$ is a morphism into the $i$-th arrow space.
\end{lemma}

\begin{proof} It is clear that $\alpha\in\aff_i(N)$ implies that $h$ is a morphism.
For the other direction assume that $h$ is a morphism.
Let $c\in C^n(N)$ be such that $n\geq i$. Let $F\subset\{0,1\}^n$ be the $n-i$ dimensional face with $0$'s in the last $i$ coordinates. Using the symmetries of cubes it is enough to show that for this particular face $\alpha^F(c)\in C^n(N)$.

Let $Q=\{0,1\}^{n-i}\times\{-1,0,1\}^i=\{0,1\}^{n-i}\times T_i$, let $f_1$ be the identity on $\{0,1\}$ and $f_2$ be the function with $f_2(-1)=1,f_2(0)=0,f_2(1)=0$. Let $f=f_1^{n-i}\times f_2^i$. The function $h=c\circ f$ is a morphism from $Q$ to $N$. Let $h'$ be the function obtained from $h$ by applying $\alpha$ to the values on $\{0,1\}^{n-i}\times 1^{i}$.

It is easy to see from our assumption that $h'$ is also a morphism to $N$.
On the other hand by lemma \ref{simpglue} the restriction of $h'$ to $\{0,1\}^{n-i}\times\{-1,1\}^i$ is a morphism to $N$. This restriction is equal to $\alpha^F(c)$.
\end{proof}

Note that definition \ref{transdef} implies that translations preserve cubes.
Recall that two cubes in $C^n(N)$ are called equivalent if they are two opposite faces of a cube in $C^{n+1}(N)$.
It is clear that a map $\alpha$ is a translation if and only if $\alpha(c)$ is equivalent with $c$ for every cube $c\in C^n(N)$. The next lemma shows a strengthening of this fact for $k$-step nilspaces.

\begin{lemma}\label{transchar2} Let $N$ be a $k$-step nilspace. An arbitrary map $\alpha:N\rightarrow N$ is a in $\aff_i(N)$ if and only if for every $c\in C^k(N)$ we have that $(c,\alpha(c))_i\in C^{k+i}(N)$.
\end{lemma}

\begin{proof} Let $c\in C^n(N)$ be an arbitrary cube and let $c'=(c,\alpha(c))_i$. By lemma \ref{transmorph} it is enough to prove that $c'\in C^{n+i}(N)$. formed by $c$ and $\alpha(c)$ as two faces. Using lemma \ref{cubechar} it is enough to show that $c'$ restricted to $k+1$ dimensional faces in $\{0,1\}^n$ with at least one point with $0$ in the last coordinate are cubes. This follows immediately from the condition of the lemma.
\end{proof}

\begin{lemma}\label{loctrans2} Let $N$ be a $k$-step nilspace. Then translations restricted to $\sim_{k-1}$ classes are local translations.
\end{lemma}

\begin{proof} It follows from lemma \ref{sim1} that if $x\sim_{k-1} y$ then $\alpha(x)\sim_{k-1}\alpha(y)$. Lemma \ref{loctrans} shows that if the $\sim_{k-1}$ classes of $x$ and $\alpha(x)$ are $F_1$ and $F_2$ then $\alpha(x+a)=\alpha(x)+a$ for an arbitrary element $a$ in the structure group $A_k$.
\end{proof}

\begin{lemma} If $N$ is a $k$-step nilspace then $\aff_1(N)$ (and thus also $\aff_i(N)$ for every $i\geq 1$) is a group.
\end{lemma}

\begin{proof} By induction on $k$ and using lemma \ref{loctrans2} we get that translations are invertible transformations. We need to show that the inverse of a translation $\alpha$ is again a translation. We go by induction on $k$. Assume that we have the statement for  $k-1$. Then in particular we have that the image of a $k$ dimensional cube $c$ under $\alpha^{-1}$ is a cube modulo $\sim_{k-1}$. This means by lemma \ref{sim2cor} that $\alpha^{-1}(c)$ is also in $C^k(N)$. Since $\alpha(\alpha^{-1}(c))=c$ we obtain that $(\alpha^{-1}(c),c)\in C^{k+1}(N)$. By lemma \ref{transchar2} applied with $i=1$ the proof is complete.
\end{proof}

\subsection{Translation bundles}

Let $N$ be a $k$-step nilspace and let $\alpha$ be an element in $\aff_i(\mathcal{F}_{k-1}(N))$.
We say that $\alpha$ can be lifted to $\aff_i(N)$ if there is an element $\alpha'\in\aff_i(N)$ such that $\pi_{k-1}(\alpha'(n))=\alpha(\pi_{k-1}(n))$ holds for every $n\in N$.
Recall that $\pi_{k-1}$ is the projection to $\mathcal{F}_{k-1}(N)$.
Our goal is to understand when can $\alpha$ be lifted this way.
We introduce a nilspace whose algebraic properties decide if there is such a lift or not.

\medskip

Let $\mathcal{T}=\mathcal{T}(\alpha,N,i)$ be the set of pairs $(x,y)\in N^2$ where $\alpha(\pi_{k-1}(x))=\pi_{k-1}(y)$. We interpret $\mathcal{T}$ as a subset of the $i$-th arrow space over $N$.
It is easy to see that if $k\geq i+1$ then $\mathcal{T}$ is an ergodic nilspace with the inherited cubic structure.

We define $\mathcal{T}^*$ as $\mathcal{F}_{k-1}(\mathcal{T})$.
We will use the next two algebraic properties of $\mathcal{T}^*$.

\begin{enumerate}
\item The group $A_k\times A_k$ acts on the space $\mathcal{T}$ by $$(x,y)\mapsto(x+a_1,y+a_2).$$ This action induces an action of $A_k$ on $\mathcal{T}^*$. For $a_1,a_2\in A_k$ we have that $(x+a_1,y+a_2)\sim_{k-1}(x,y)$ if and only if $a_1=a_2$. It follows that the elements of $\mathcal{T}^*$ represent local translations $\phi:F_1\rightarrow F_2$ where $F_1,F_2$ are $\sim_{k-1}$ classes in $N$ with $\alpha(F_1)=F_2$.
\item The map $(x,y)\mapsto x$ creates a map $\mathcal{T}\rightarrow N$. It induces a map $\gamma:\mathcal{T}^*\rightarrow \mathcal{F}_{k-1}(N)$.
\end{enumerate}

Combining these two facts one can see easily that $\mathcal{T}^*$ is a degree $k-i$ extension of $\mathcal{F}_{k-1}(N)$ by $A_k$.

\begin{proposition}\label{transext} Let $N$ be a $k$-step nilspace and $\alpha\in\aff_i(\mathcal{F}_{k-1}(N))$. If $\mathcal{T}^*=\mathcal{T}^*(\alpha,N,i)$ is a split extension then $\alpha$ lifts to an element $\beta\in\aff_i(N)$.
\end{proposition}

\begin{proof} Let $\gamma':\mathcal{F}_{k-1}(N)\rightarrow\mathcal{T}^*$ be a morphism such that $\gamma\circ\gamma'$ is the identity map. The element $\gamma'(\pi_{k-1}(x))$ in $\mathcal{T}^*$ represents a local translation from the $\sim_{k-1}$ class $F_1$ of $x$ to the class $\alpha(F_1)$. Let $\beta(x)$ denote the image of $x$ under this local translation.
We claim that the map $\beta$ is in $\aff_i(N)$.
Let $h:N\rightarrow N\times N$ be the map defined by $h(n)=(n,\beta(n))$.
According to lemma \ref{transchar2} it is enough to show that for every $c\in C^k(N)$ we have that $h\circ c$ is a cube in the $i$-th arrow space on $N\times N$. Since $\gamma'$ is a morphism we have that $\gamma'(\pi_{k-1}(c))$ is in $C^k(\mathcal{T}^*)$. By lemma \ref{sim2cor} we obtain that any lift of $\gamma(\pi_{k-1}(c))$ to $\mathcal{T}$ is in $C^k(\mathcal{T})$. The pairs $\{(c(v),\beta(c(v)))|v\in\{0,1\}^k\}$ form such a lift. This shows that $h\circ c$ in a cube in $\mathcal{T}$.
\end{proof}

The condition of lemma \ref{transext} holds for $\alpha$ if and only if $\mathcal{T}_0(\alpha,N)$ is a split extension.
A way of checking the condition is to show that the cocycle describing $\mathcal{T}_0(\alpha,N)$ as an extension of $\mathcal{F}_{k-1}$ by $A_k$ is a coboundary.

\subsection{Nilpotency}

Let $N$ be a $k$-step nilspace. In this part we investigate the properties of the groups $\aff_i(N)$.
The main idea is borrowed from the paper \cite{HKr} by Host and Kra.

\begin{lemma} We have that $[\aff_i(N),\aff_j(N)]\subseteq\aff_{i+j}(N)$.
\end{lemma}

\begin{proof} Let $F$ be a face in $\{0,1\}^n$ of codimension $i+j$. Then $F=F_1\cap F_2$ where $F_1$ is a face of codimension $i$ and $F_2$ is a face of codimension $j$.
Assume that $\alpha_1\in\aff_i(N)$ and $\alpha_2\in\aff_j(N)$. Then $[\alpha_1^{F_1},\alpha_2^{F_2}]=[\alpha_1,\alpha_2]^F$. This implies that if $c\in C^n(N)$ then $[\alpha_1,\alpha_2]^F(c)\in C^n(N)$.
\end{proof}

\begin{corollary} The group $\aff(N)$ is $k$-nilpotent and $\{\aff_i(N)\}_{i=1}^{k+1}$ is a central series in it.
\end{corollary}

\begin{lemma} if $k\geq i$ then the action of $A_k$ is in $\aff_i(N)$.
\end{lemma}

\begin{proof} It follows directly from theorem \ref{bundec}.
\end{proof}

\begin{definition} We say that two cubes $c_1,c_2\in C^n(N)$ are translation equivalent if $c_2$ can be obtained from $c_1$ be a sequence applications of operations $\alpha^F$ where $\alpha\in\aff_i(N)$ and $F$ is a face in $\{0,1\}^n$ of codimension $i$. Note that the number $i$ can be different in the above operations.
A cube is called translation cube if it is translation equivalent with a constant cube.
\end{definition}

\section{Compact nilspaces}

In this part of the paper we study compact topological versions of nilspaces.

\begin{definition} A nilspace $N$ is called {\bf compact} if it is a second countable, compact, Hausdorff topological space and $C^n(N)$ is a closed subset of $N^{\{0,1\}^n}$ for every $n\in\mathbb{N}$. 
\end{definition}

An important consequence of compactness is that $\mathcal{F}_k(N)$ is compact for every $k\in\mathcal{N}$. Furthermore all the abelian groups occurring in theorem \ref{bundec} are compact abelian groups. This is proved in the following section.

When checking the compactness of a $k$-step nilspace the next lemma is useful.

\begin{lemma}\label{kcompchar} A $k$-step nilspace $N$ is compact if and only if $N$ is a compact, Hausdorff, second countable topological space and $C^{k+1}(N)$ is a closed subset of $N^{\{0,1\}^{k+1}}$. 
\end{lemma}

\begin{proof} First assume that $n>k+1$. For every $k+1$ dimensional face $F$ let $Q_F$ denote the set of functions $f:\{0,1\}^n\rightarrow N$ whose restriction to $F$ is a cube. By lemma \ref{cubechar} we have that $C^n(N)$ is the intersection of the sets $Q_F$ where $F$ runs through all the $k+1$ dimension faces of $\{0,1\}^n$.
If $n<k+1$ then $C^n(N)$ is the projection of $C^{k+1}(N)$ to $N^{\{0,1\}^n}$ and thus it is closed. 
\end{proof}

\subsection{Continuity of the bundle decomposition}\label{contbundec}

\medskip

In the previous chapters we associated a number of structures to an abstract nilspace. Here we study compactness and continuity in these structures.
By an abuse of language we will use the term \emph{compact space} for a compact, Hausdorff, second countable topological space. We can immediately observe that if $N$ is a compact nilspace then the arrow space, the higher degree arrow spaces and $\partial_x N$ are all naturally compact. Similarly if $A$ is a compact abelian group then $\mathcal{D}_k(A)$ is a compact nilspace since it is defined by linear equations.
Other constructions, especially those involving taking quotients by equivalence relations, will require a more subtle treatment. 
For those we will use the following elementary facts.

\begin{enumerate}
\item If $f:X\rightarrow Y$ is a map between compact spaces. Then $f$ is continuous if and only if its graph is closed. 
\item Compact spaces are metrizable.
\item If $X$ is a compact space an $\sim$ is a closed equivalence relation on $X$ (i.e. $\{(x,y)|x\sim y\}\subset X\times X$ is closed) then $X/\sim$ with the quotient topology is also compact.  
\end{enumerate}

As a first demonstration we show the following.

\begin{lemma}\label{topkfolderg} If $N$ is a $k$-step $k$-fold ergodic compact nilspace then $N$ is isomorphic as a compact nilspace to $\mathcal{D}_k(A)$ for some compact abelian group $A$. 
\end{lemma}

\begin{proof} We start with the one step case. We use the notation of lemma \ref{onestep}. We need to ceck the continuity of addition and inverse.
The graph of addition can be written as 
$$\{(x_1,x_2,x_3,x_4)\in C^2(N)~|~x_1=e\}$$
and the graph of the inverse is similarly defined by the equation $x_1=x_4=e$.
These are closed sets so the operations are continuous.

To see the general case recall that the group structure on $N$ in the abstract case was recovered from the one step nilspace $\partial_e^{k-1}(N)$ which is now compact.
\end{proof}

\begin{lemma} Let $N$ be a compact nilspace and $k\in\mathbb{N}$. Then the factor $\mathcal{F}_k(N)$ with the quotient topology is a compact nilspace.
\end{lemma}

\begin{proof} The proof boils down to checking that $\sim_k$ is a closed equivalence relation. This is true because the set of cubes satisfying the condition in lemma \ref{sim1} is closed.
\end{proof}

In the remaining part of this chapter we explain how to turn abstract nilspace theory into compact nilspace theory.
We start with the definition of a compact abelian bundle. This is an abstract $A$-bundle $T$ with base $S$ in which $A$, $T$ and $S$ are compact spaces such that $A$ is a compact group and the action $\alpha:A\times T\rightarrow T$ is continuous. Furthermore $S$ has the quotient topology.
Notice that we are not assuming that the bundle is locally trivial but if $A$ is finite dimensional then this automatically holds \cite{Gl}.
From this definition it is clear how to define a $k$-fold compact abelian bundle.
The definition of a compact degree-$k$ bundle is as in the abstract case but with a compact $k$-fold abelian bundle that is a compact cubespace for the same topology.

\begin{lemma}\label{topbundec} A compact cubespace $N$ is a compact degree-$k$ bundle if and only if $N$ is a $k$-step compact nilspace. 
\end{lemma}  

\begin{proof} It is clear that a compact degree-$k$ bundle is a $k$-step compact nilspace since as it was shown in the proof of theorem \ref{bundec} it satisfies the nilspace axioms. 
For the other direction we need to verify that the structure groups $\{A_i\}_{i=1}^k$ can be (uniquely) given the structure  of compact topological groups such that the action of $A_i$ on $\mathcal{F}_i(N)$ is continuous.   
Using induction on $k$ it is enough to check this statement for $A_k$.
Let $F$ be a $\sim_{k-1}$ class of $N$. We learned from the proof of theorem \ref{bundec} that the cubic structure restricted to $F$ is a $k$-fold ergodic $k$-step nilspace. Lemma \ref{sim1} implies that $F$ is closed and so by lemma \ref{topkfolderg} we obtain that there is a compact abelian group structure on $A_k$ such that $F$ is isomorphic to $\mathcal{D}_k(A_k)$ as a compact nilspace. The continuity of the action of $A_k$ is equivalent with the fact that its graph $\{(a,y,y+a)~|~a\in A,y\in N\}$ is closed in $A\times N^2$.  
On the other hand this set can be described (as seen in the proof of theorem \ref{bundec}) using the arrow space $N'$ of $N$; it is essentially the set $\{(u,v) \in N' \times N' ~|~ u \sim_{k-1} v, u \in \{e\} \times N \subset N' \}$ where $e \in F$ is fixed. Note that $u \sim_{k-1} v$ implies the first coordinate of $v$ lies in $F$, and thus can be written as $e+a$ for some $a\in A$. Then, as shown in the proof of theorem \ref{bundec}, the second coordinates of $u$ and $v$ are of the form $y$ and $y+a$.
Now that we know that the action is continuous, the uniqueness of the topology on $A_k$ follows since it is homeomorphic to the induced topology on any orbit of the action.
\end{proof}

The notions of sub-bundle, bundle-morphism, totally surjective bundle morphism, and kernel from chapter \ref{subbundchap} can be easily transfered to the setting of compact $k$-fold abelian bundles with the following additions. First of all we require that all the maps used in the definitions are continuous. To see that the kernel $K$ of a bundle morphism is compact observe that $K_i$ is a closed subset of $T_i\times T_k'$.

We continue with lemma \ref{restbund}. In the topological version we assume that $N$ is a compact nilspace and $P$ is a finite cubespace. Note that the restricted homomorphism set $\Hom_f(P,N)$ is a closed subset of $N^P$ and thus it is compact. To see that $\Hom_f(P,N)$ is a $k$-fold bundle in the topological sense we only need to observe that $\Hom_{X\rightarrow 0}(P,\mathcal{D}_i(A_i))$ is a closed subgroup in $A_i^P$. 

\subsection{Haar measure on abelian bundles and nilspaces}\label{haar}

Compact $k$-step nilspaces are generalizations of compact abelian groups. It will be important to generalize the normalized Haar measure to them. Recall that the normalized Haar measure is a shift invariant Borel probability measure. Such measures always exist on compact groups and they are unique.

First we define the Haar measure for compact abelian bundles. Let $T$ be a compact $A$-bundle  over a space $S$ and action $\alpha:A\times T\rightarrow T$.  Assume that $S$ has a Borel probability measure $\mu_S$. Then we introduce the extension $\mu$ of $\mu_S$ as the unique Borel probability measure on $T$ which is $A$ invariant.
The measure $\mu$ can be defined through the property that
\begin{equation}\label{mesdef}
\mu(H)=\int_{s\in S}\mu_A(\pi_S^{-1}(s)\cap H)~d\mu
\end{equation}
where $H$ is a Borel set of $T$, $\pi$ is the projection to $S$ and $\mu_A$ is the Haar measure on $A$ and any fiber of $\pi$.

We define the Haar measure on a compact $k$-fold abelian bundle iteratively. If it is already defined for $k-1$ fold bundles then we use (\ref{mesdef}) to extend it from the factor $T_{k-1}$ to $T_k$.
We use theorem \ref{bundec} to define (normalized) Haar measures for $k$-step nilspaces.

By abusing the notation we will always denote the Haar measure by $\mu$. Since we never define two different measure on one structure it will not cause any problem.

The following fact is well known for compact abelian groups.

\begin{lemma}\label{abpres} Surjective continuous (affine) homomorphisms between compact abelian groups are measure preserving.
\end{lemma}

We will need a generalization of this fact for $k$-fold compact abelian bundles.

\begin{lemma}\label{bundpres} Let $\phi:T\rightarrow T'$ be a totally surjective continuous map between two compact $k$-fold abelian bundles. Then $\phi$ preserves the Haar measure. This means that for an arbitrary Borel set $H\subset T'$ we have $\mu(H)=\mu(\phi^{-1}(H))$.
\end{lemma}

\begin{proof} The proof is an induction using lemma \ref{abpres}. The map $\phi$ induces a map $\phi'$ from $T_{k-1}$ to $T'_{k-1}$. If we know the statement for $k-1$ then $\phi'$ is measure preserving. On the other hand it is measure preserving on the fibers so the integral in (\ref{mesdef}) is preserved.
\end{proof}

The next lemma follow directly form lemma \ref{bundpres} and lemma \ref{morphbund}

\begin{lemma} Continuous fiber surjective morphisms between $k$-step nilspaces are measure preserving.
\end{lemma}

\subsection{Continuous systems of measures}\label{csmchap}

In section \ref{haar} we constructed a measure on a compact space with an Abelian
bundle structure, this measure combined some given measure on the base of the bundle with
the Haar measure on each fiber. We will need later on that these measure space fibers vary
continuously in a certain sense.To express this, we will use the concept of a \emph{continuous
system of measures (CSM)}, which consists of
\begin{enumerate}
  \item A continuous map $\pi : X \to Y$ between compact Hausdorff topological spaces, and
  \item a family of measures $\{ \lambda_y \}_{y \in Y}$ such that
  \begin{enumerate}
    \item the measure $\lambda_y$ is concentrated on the fiber $\pi^{-1}(y)$, and
    \item\label{contprop} for any continuous real-valued function $f$ on $X$, the function
    \[ y \mapsto \int_{\pi^{-1}(y)} f \; d\lambda_y \] is continuous.
  \end{enumerate}
\end{enumerate}

Here, condition \ref{contprop} is the main one, as it expresses the continuity of the family
$\{\lambda_y\}$ of measures. A good reference for continuous systems of measures is
\cite{AD}, which develops them in slightly greater generality (using locally compact spaces,
rather than compact ones, for which \ref{contprop} is only taken to hold for compactly
supported functions), as well as developing the notion of Borel systems of measures, an
analogue for Borel maps $\pi : X \to Y$. That paper also contains references to earlier
literature on these and related notions.

Given a continuous systems of measures $\{\lambda_y\}$ for a map $\pi : X \to Y$ and
a measure $\mu$ on the base $Y$, one can define a measure on $X$ by $\lambda(E) = \int_Y \lambda_y(E \cap \pi^{-1}(y)) d\mu$, just as we did earlier for the special case of an Abelian bundle. To show this really is a special case, we have the following result:

\begin{lemma} Given any compact  Abelian $A$-bundle $T$ over a base $S$, the family of measures consisting of copies of the Haar measure on $A$ is a CSM.
\end{lemma}

\begin{proof}
We just need to prove the continuity property. Let $f : T \to \mathbb{R}$ be a continuous
function. We need to show that the function $g : S \to \mathbb{R}$,
$g(s) =\int_{\pi^{-1}(s)} f(t) \; d\mu_A$
is continuous (here, $\mu_A$ denotes Haar measure on $A$ and any fiber of $T \to S$).
Since $S$ has the quotient topology it is enough to show that $g \circ \pi$ is continuous.
Now, $g(\pi(t)) = \int_{a\in A} f(t+a) \; d\mu_A$ and the function $T \times A \to \mathbb{R}$
given by $(t,a) \mapsto f(t+a)$ is continuous, so the result follows from the following simple
claim:

\begin{claim} Let $X$ and $Y$ be compact spaces and $\mu$ be a Borel probability
 measure on $Y$.
Then the family of measures on $X\times Y$ consisting of copies of $\mu$ is a CSM.
\end{claim}

Let $f : X \times Y \to \mathbb{R}$ be a continuous function. Then $f$ can be approximated
arbitrarily well in $L^\infty(X \times Y)$ by linear combinations of rank 1 functions, i.e.,
functions of the form $(x,y) \mapsto f_1(x)f_2(y)$. For a rank 1 function, it is clear that
integrating out $y$ leaves a continuous function of $x$, so we get that $\int_{y \in Y} f(x,y) \; d\mu$
is a uniform limit of continuous functions and thus continuous itself.
\end{proof}

\begin{lemma}\label{relcsm} If $T$ is a relative compact $k$-fold Abelian bundle then the Haar measures on the fibers form a CSM. 
\end{lemma}

\begin{proof} Notice that each fiber of the relative bundle is a $k$-fold compact Abelian bundle in its own right and thus has a Haar measure as defined in chapter \ref{haar}. 
Let $T_0, T_1, T_2, \ldots, T_k = T$ be the factors of $T$. By the previous lemma, each $T_i$ is a CSM over $T_{i-1}$. The CSM we want to define on $T$ over $T_0$ is just the composition of these in the following sense:

Let $\{\lambda_y\}_{y \in Y}$ be a CSM on $\pi : X \to Y$ and let $\{\mu_z\}_{z \in Z}$ be another CSM on $\tau : Y \to Z$.  The composition is a CSM on $\tau \circ \pi : X \to Z$ with measures given by
\[ \nu_z(E) = \int_{y \in \tau^{-1}(z)} \lambda_y(E \cap \pi^{-1}(y)) \; d\mu_z. \]
For a proof that this defines a CSM see \cite{AD}.
\end{proof}

\medskip

Let $S,T$ be a pair of compact spaces and assume that $S$ is a probability space.  We denote by $L(S,T)$ the set of Borel measurable functions up to $0$ measure change. Let $\{\mu_y\}_{y\in Y}$ be a CSM on the fibers of $\pi:X\to Y$. Let $\mathcal{L}(X,T)=\cup_y L(\pi^{-1}(y),T)$. The projection $\tilde{\pi}:\mathcal{L}(X,T)\rightarrow Y$ is defined by $\tilde{\pi}(f)=y$ if $f\in L(\pi^{-1}(y),T)$. We define a topology on $\mathcal{L}(X,T)$ as the weakest topology in which the following functions are continuous:
$$f\mapsto \int_{x\in \pi^{-1}(\tilde{\pi}(f))}F_1(f(x))F_2(x)~d\mu_{\tilde{\pi}(f)}$$
where $F_1:T\rightarrow\mathbb{C}$ and $F_2:X\rightarrow\mathbb{C}$ are continuous functions. Notice that $F_2$ is defined and continuous on the whole space $X$, this is what ties the fibers together. With this topology $\mathcal{L}(X,T)$ becomes a Polish space.

Note that the space $\mathcal{L}(X,T)$ is a generalization of the $L^1$ topology on bounded measurable functions. To be more precise take $T$ to be the closed unit ball in $\mathbb{C}$. Then for every fixed $y\in Y$, convergence of functions in $L(\pi^{-1}(y),T)$ (as a subspace of $\mathcal{L}(X,T)$) is the same as convergence in the usual $L^1$ topology.
The implication that $L^1$ convergence implies convergence in $\mathcal{L}(X,T)$ is left to the reader. 
To see the other direction let $\{f_i\}_{i=1}^\infty$ be a sequence converging in $\mathcal{L}(X,T)$ to $f\in L(\pi^{-1}(y),T)$. It is clear from the definitions that the weak limit of $\{f_i\}_{i=1}^\infty$ is $f$ in the Hilbert space $L^2(\pi^{-1}(y),\mu_y)$. Furthermore $\lim_{i\to\infty}\|f_i\|_2=\|f\|_2$.
 It is well known that these two facts imply that $f$ is the $L^2$ limit of $\{f_i\}_{i=1}^\infty$. Since the $L^2$ topology is equivalent with the $L^1$ topology for functions with values in $T$ the proof is complete.

The space $T$ in $\mathcal{L}(X,T)$ will most often be a compact abelian group. In that case $\mathcal{L}(X,T)$ has an action of $T$ by translation. Then next lemma says that this action is continuous.

\begin{lemma}\label{contact} Let $A$ be a compact abelian group. For any continuous sytem of measures the space $\mathcal{L}(X,A)$ defined above has a continuous $A$ action $\alpha:A\times\mathcal{L}(X,A)\rightarrow\mathcal{L}(X,A)$ defined by $\alpha(g, f)(x) = f(x) + g$.
\end{lemma}

\begin{proof}
It is enough to show that for arbitrary continuous functions $F_1 : A \to \mathbb{C}$ and $F_2 : X \to \mathbb{C}$, the function $$(g, f) \mapsto \int_{x \in \pi^{-1}(\tilde{\pi}(f))} F_1(f(x)+g) F_2(x) \; d\mu_{\tilde{\pi}(f)}$$ is continuous. We again approximate by rank one functions. Consider the function $A \times A \to \mathbb{C}$ defined by $(g_1,g_2) \mapsto F_1(g_1 + g_2)$, and approximate it in $L^\infty(A \times A)$ by linear combinations of functions of the form $(g_1, g_2) \mapsto H_1(g_1)H_2(g_2)$ where $H_1$ and $H_2$ are continuous. For such a rank one function, the continuity follows from the definition of CSM. 
\end{proof}

To state the next lemma we need some notation. Assume that $\{\mu_y\}_{y\in Y}$ is a CSM on the fibers of $\pi:X\rightarrow Y$. We denote by $X\times_Y X$ the compact space $\{(a,b)|a,b\in X,\pi(a)=\pi(b)\}$. There is a natural projection $\pi':X\times_Y X\rightarrow Y$ defined by $\pi'(a,b)=\pi(a)=\pi(b)$ and a CSM structure defined by $\{\mu_y\times\mu_y\}_{y\in Y}$ on $X\times_Y X$. 

\begin{lemma}\label{diffcont} Let $\{\mu_y\}_{y\in Y}$ be a CSM on the fibers of $\pi:X\to Y$ and let $A$ be a compact abelian group,  Let $\mathcal{E}:\mathcal{L}(X,A)\rightarrow\mathcal{L}(X\times_Y X,A)$ denote the operator which maps a function $g$ to the function $\mathcal{E}(g)(a,b)=g(a)-g(b)$. Then $\mathcal{E}$ is a quotient map from $\mathcal{L}(X,A)$ to its image.
\end{lemma}

\begin{proof} The continuity of $\mathcal{E}$ is easy and it is left to the reader.
To check the quotient map property it is enough to show that if for a sequence of functions $\{f_i\}_{i=1}^\infty$ in $\mathcal{L}(X,A)$ the sequence $\{\mathcal{E}(f_i)\}_{i=1}^\infty$ is convergent then there is a convergent subsequence  $\{f_{m_i}\}_{i=1}^\infty$.

Assume that $\{\mathcal{E}(f_i)\}_{i=1}^\infty$ has limt $\mathcal{E}(f)$ for some $f\in\mathcal{L}(X,A)$. Furthermore assume that $f_i$ is defined on $\pi^{-1}(y_i)$ and $f$ is defined on $\pi^{-1}(y)$ where $\{y_i\}_{i=1}^\infty$ is a sequence in $Y$ converging to $y$. Let $\mathcal{C}$ denote the closed unit ball in $\mathbb{C}$. For a character $\chi\in\hat{A}$ and $\delta>0$ let $\chi_\delta:X\rightarrow\mathcal{C}$ be a continuous function on $X$ whose restriction to $\pi^{-1}(y)$ is at most $\delta$-far from $\chi\circ f$ in $L^1$ according to $\mu_y$.

We claim that for every $\chi\in\hat{A}$ and $\epsilon>0$ there are values $\delta>0$ and $n_{\chi,\epsilon}\in\mathbb{N}$ such that if $j>n_{\chi,\epsilon}$ then the function $(\chi\circ f_j)\overline{\chi_\delta}$ is at most $\epsilon$ far in $L^1$ from some constant function on $\pi^{-1}(y_j)$.

Let $\mathcal{E}':\mathcal{L}(X,\mathcal{C})\rightarrow\mathcal{L}(X\times_Y X,\mathcal{C})$ denote the operator with $\mathcal{E}'(g)(a,b)=g(a)\overline{g(b)}$. It is clear that $\mathcal{E}'(g_1g_2)=\mathcal{E}'(g_1)\mathcal{E}'(g_2)$ and that $\mathcal{E}'(\chi\circ g)=\chi\circ\mathcal{E}(g)$ for every $\chi\in\hat{A}$.
If $\delta$ is small enough then $\mathcal{E}'((\chi\circ f)\overline{\chi_\delta})$ is arbitrarily close to the constant $1$ function. It follows from the assumption of the lemma that if $j$ is big enough that $\mathcal{E}'((\chi\circ f_j)\overline{\chi_\delta})$ is also close to the constant $1$ function. Clearly, it is only possible if $(\chi\circ f_j)\overline{\chi_\delta}$ is close to a constant function (and this constant has absolute value close to one). 

Since there are at most countably many elements in $\hat{A}$, from a standard (iterated) diagonalization argument (using the above claim with smaller and smaller $\epsilon$ for each character) it follows that one can chose a growing sequence $\{m_i\}_{i=1}^\infty$ in $\mathbb{N}$ such that for every $\chi\in\hat{A}$ there is a constant $c_\chi\in\mathbb{C}$ of length $1$ such that the sequence $\{\chi\circ f_{m_i}\}_{i=1}^\infty$ converges to $c_\chi (\chi\circ f)$. Furthermore it is clear that the function $\chi\mapsto c_\chi$ is a homomorphism. It follows that there is an element $t\in A$ such that $\chi(t)=c_\chi$ holds for every $\chi\in\hat{A}$. We obtain that $\{f_{m_i}\}_{i=1}^\infty$ converges to $f+t$.
\end{proof}

We will need the following technical lemma.

\begin{lemma}\label{csmtechnical} Let $\{\mu_y\}_{y\in Y}$ be a CSM on the fibers of $\pi:X\to Y$. Let $K$ be a compact space with Borel measure $\nu$. Assume that $f:X\rightarrow K$ is continuous and that the restriction of $f$ to the fibre $\pi^{-1}(y)$ is measure preserving for every $y\in Y$.  Let $T$ be a compact space and $g:K\rightarrow T$ be a Borel function. 
Let $h:Y\rightarrow\mathcal{L}(X,T)$ be the map such that $h(y)$ is the restriction of $g\circ f$ to $\pi^{-1}(y)$.
Then $h$ is continuous. 
\end{lemma}

\begin{proof} Let $F_1:T\rightarrow\mathbb{C}$ and $F_2:X\rightarrow\mathbb{C}$ be continuous functions.
Let $q:Y\rightarrow\mathbb{C}$ be the function defined by
$$y\mapsto\int_{x\in\pi^{-1}(y)}F_1(g(f(x)))F_2(x)~d\mu_y.$$
We have to show that $q$ is continuous.
Let $\epsilon>0$ be arbitrary. It is well known that there is a continuous approximation $F_3:K\rightarrow\mathbb{C}$ of $F_1\circ g$ such that $\|F_1\circ g-F_3\|_1\leq\epsilon$.
Let $q':Y\rightarrow\mathbb{C}$ be the function define by 
$$y\mapsto\int_{x\in\pi^{-1}(y)}F_3(f(x))F_2(x)~d\mu_y.$$
By the measure preserving property of $f$ we obtain that $|q(y)-q'(y)|\leq\epsilon\|F_2\|_\infty$ holds for every $y\in Y$.
Since $q'$ is continuous and we have such an approximation for every $\epsilon>0$ the proof is complete.
\end{proof}

\subsection{Probability spaces of morphisms}\label{morpro}

Lemma \ref{restbund} says that morphism sets and more generally restricted morphism sets are often iterated bundles and thus in the compact case they have a probability space structure coming from Haar measure. A few concrete examples of such probability spaces will be crucial in our arguments.

\begin{definition} Let $P$ be a cube space and assume that $X,Y$ are two subsets in $P$. We say that $X,Y$ is a good pair if $X$ and $X\cap Y$ have the extension property (with the inherited cubespace structure) and every morphism $f:Y\rightarrow\mathcal{D}_k(A)$ with $f|_{X\cap Y}=0$ extends to a morphism $f:P\rightarrow\mathcal{D}_k(A)$ with $f|_X=0$.
\end{definition}

\begin{lemma}\label{pairlem} Let $P\subseteq\{0,1\}^n$ be a sub-cubespace with the extension property and $X,Y\subset P$ be a good pair in $P$. Then for every nilspace $N$ and morphism $f:X\rightarrow N$ we have that the restriction map to $Y$
$$\phi:\Hom_f(P,N)\rightarrow\Hom_{f|_{X\cap Y}}(Y,N)$$ is a totally surjective bundle morphism.
\end{lemma}

\begin{proof} Assume that $N$ is a $k$-step nilspace and by induction assume that the statement is verified for $k-1$ step nilspaces. By lemma \ref{restbund} It is enough to show that the map $\Hom_{X\rightarrow 0}(P,\mathcal{D}_k(A_k))\rightarrow\Hom_{X\cap Y\rightarrow 0}(Y,N)$ is surjective. This follows directly form the definition of good pairs.
\end{proof}

Note that lemma \ref{bundpres} implies that if $N$ is a compact finite step nilspace then the restriction map in lemma \ref{pairlem} is measure preserving.

\begin{lemma}\label{pairlem2} If $P\subseteq\{0,1\}^n$ has the extension property and $X,Y\subset P$ is a good pair of sub-cubespaces then $X\cup Y$ (with the union of their cubic structures) has the property that any morpism $f:X\cup Y\rightarrow N$ into a finite step nilspace $N$ extends to a morphism $f':P\rightarrow N$.
\end{lemma}

\begin{proof} The proof is an induction on the number of steps of $N$. The statement is trivial for $0$ step nilspaces. Assume that it is true for $k-1$ step nilspaces and let $N$ be a $k$-step nilspace.
Let $f:X\cup Y\rightarrow N$ be a morphism and let $f_2:P\rightarrow\mathcal{F}_{k-1}(N)$ be an extension of $f$ modulo $\sim_{k-1}$. Using lemma \ref{lifting} we can find a morphism $f_2':P\rightarrow N$ such that $f_2'\equiv f_2$ modulo $\sim_{k-1}$. Let $g=f'_2|_X-f|_X$. Acording to our assumption there is an extension $f_3:P\rightarrow\mathcal{D}_k(A_k)$ of $g$. Let $g_2=f_2'-f_3$. We have that $g_2|_X=f|_X$. Now let $g_3$ be an extension of $g_2|_Y-f|_Y$ to $P$ with $g_3|_X=0$. Then $f'=g_2-g_3$ is an extension of $f$ to $P$.
\end{proof}

\noindent{\bf Construction 1.}~~Let $N$ be a $k$-step nilspace and $f:N\rightarrow N'$ be a fiber surjective morphism into another $k$-step nilpace. According to lemma \ref{restbund} if $n\in\mathbb{N}$ then $C^n(N)=\Hom(\{0,1\}^n,N)$ is a probability space. We call this distribution the uniform distribution on $C^n(N)$. Furthermore by lemma \ref{collection} the map $f$ induces a measure preserving map from $C^n(N)$ to $C^n(N')$. The fibers of this map also have a $k$-fold bundle structure and so they are all probability spaces.  
It is trivial that every one element set in a cubespace has the extension property. As a consequence we have that if $x\in N$ is an arbitrary element then we can view $C_x^n(N)=\Hom_{0^n\rightarrow x}(\{0,1\}^n,N)$ as a probability space. Moreover, the measures on the $C_x^n(N)$ vary continuously with $x$.

\begin{lemma} Let $\psi_0 : C^n(N) \to N$ be the restriction map defined by $\psi_0(c) = c(0)$. Then the Haar measures on the fibers $\psi_0^{-1}(x) = C_x^n(N)$ form a CSM. 
\end{lemma}

\begin{proof}
According to lemma \ref{collection}, the map $\psi_0$ is a totally surjective bundle morphism between the $k$-fold Abelian bundle structures on $C^n(N)$ on $N$. Then, the kernel of $\psi_0$ is compact relative $k$-fold bundle whose fibers are the $C_x^n(N)$ by remark \ref{kerident}. Finally lemma \ref{relcsm} gives the required CSM.
\end{proof}

\medskip

\noindent{\bf Construction 2.}~~In the three cube $T_n$ let $X=\omega(\{0,1\}^n)=\{1,-1\}^n$. Then we claim that $X$ has the extension property.  Indeed, if $f:X\rightarrow N$ is a morphism to some nilspace and $h:\{-1,0,1\}\rightarrow\{0,1\}$ is the map with $h(-1)=-1,h(0)=1,h(1)=1$ then $f'=f\circ h^n$ is an extension of $f$ to $T_n$. 

\medskip

\noindent{\bf Construction 3.}~~In the three cube $T_n$ let $Y=\Psi_{0^n}(\{0,1\}^n)=\{0,1\}^n$. Then we claim that $Y$ has the extension property. Indeed, if $f:Y\rightarrow N$ is a morphism to some nilspace and $h:\{-1,0,1\}\rightarrow\{0,1\}$ is the map with $h(-1)=1,h(0)=0,h(1)=1$ then $f'=f\circ h^n$ is an extension of $f$ to $T_n$. As a consequence, using the symmetries of $T_n$, we obtain that $\Psi_v(\{0,1\}^n)$ also has the extension property for every $v\in\{0,1\}^n$. By lemma \ref{restbund} it follows that $\Hom_f(T_n,N)$ is a probability space.

\medskip

\noindent{\bf Construction 4.}~~Let $T_n$ be the three cube, $X=\{1,-1\}^n$ and $Y=\{0,1\}^n$. We also assume that $\omega\in C^n(P)$. By lemma \ref{extvert} this modification does not change the homomorphism set of $T_n$ into any nilspace.
We show that $X,Y$ is a good pair in $T_n$. Since $X\cap Y=\{1^n\}$ is a single point the extension property is clear in this case. According to construction 3 the set $X$ has the extension property.
Let $f:Y\rightarrow\mathcal{D}_k(A)$ be a morphism with $f(1^n)=0$. Let $h(-1)=1,h(0)=0,h(1)=1$. Then $f'=f \circ h^n$ is an extension of $f$ with $f|_X=0$. This shows that $X,Y$ is a good pair in $T_n$. By lemma \ref{pairlem} we have that if $f:X\rightarrow N$ is a morpism to a finite step nilspace then the restriction map from $\Hom_f(X,N)$ to $\Hom_{1^n\mapsto f(1^n)}(Y,N)$ is measure preserving.

\medskip

\noindent{\bf Construction 5.}~~Let $T_n$ be the three cube for some $n\in\mathbb{N}$. Let $x\in N$ be an element in a nilspace $N$. Then by lemma \ref{restbund} the set $Q_x=\Hom_{1^n\mapsto x}(T_n,N)$ is a probability space. Let $v\in\{0,1\}^n$ such that $v\neq 0^n$. We claim that if $t$ is a random element of $Q_x$ then $t\circ\Psi_v$ is a uniformly random element of $C^n(N)$. To see this let $X=\{1^k\}$ and $Y=\Psi_v(\{0,1\}^n)$. We have that $X\cap Y=\emptyset$. Since $X$ has one element, it has the extension property in $T_n$. Assume that $f:Y\rightarrow\mathcal{D}_k(A)$ is a morphism. By construction 3 the map $f$ has an extension $f_2:T_n\rightarrow\mathcal{D}_k(A)$. Since $1^k\notin Y$ there is a face of $T_n$ of the form $F=\{-1,0,1\}^a\times\{1\}\times\{-1,0,1\}^b$ with $a+b=n-1$ such that $F\cap Y=\emptyset$. Let $f':T_n\rightarrow\mathcal{D}_k(A)$ be the function obtained from $f_2$ by subtracting $f_2(1^n)$ from the values on $F$. It is easy to see that $f'$ is a morphism which extends $f$ and $f|_X=0$.
Now by lemma \ref{pairlem} the claim is proved. 

\medskip

\subsection{Measurable cocycles}\label{chapmeas}

\bigskip

From now on we will always assume that $N$ is a compact $n$-step nilspace and $A$ is a compact abelian group. We will only consider measurable cocycles on $N$. After developing some formalism we will see that every $A$ valued measurable cocycle defines a compact $A$-bundle over $N$ which can again be given a compact nilspace structure. In other words, a measurable cocycle defines a continuous extension of $N$ by $A$.

Recall that for $x\in N$ we denote by $C_x^k(N)$ the set of restricted cubes in which $0$ is mapped to $x$.
 The spaces $C_x^k(N)$ are the fibers of the map $\psi_0:C^k(N)\rightarrow N$ defined by $\psi_0(c)=c(0)$.
Each space $C_x^k(N)$ is a $k$-fold abelian bundle  and consequently has its own probability space structure (see construction 1 in chapter \ref{morpro}). We denote the probability measure on $C_x^k(N)$ by $\mu_x$. As we have seen the measures $\{\mu_x\}_{x\in N}$ form a CSM. 
Let $\rho:C^k(N)\rightarrow A$ be a measurable function. We denote by $\rho_x$ its restriction to $C^k_x(N)$.
We define $\mathcal{L}_k(N,A)$ as $\mathcal{L}(C^k(N),A)$ using the CSM with projection $\psi_0:C^k(N)\rightarrow N$.
Recall that $\mathcal{L}$ was defined in chapter \ref{csmchap}.

\begin{proposition}\label{meascont} Let $\rho:C^k(N)\rightarrow A$ be a measurable cocycle of degree $k-1$. Then $$M=\{\rho_x+a~|~x\in N, a\in A\}\subset\mathcal{L}_k(N,A)$$ is a compact $A$ bundle over $N$ with projection $\pi$.  
\end{proposition}

\begin{proof} The action of $A$ on $M$ given by $f\rightarrow f+a$ is continuous by lemma \ref{contact}. It is enough to prove that $M$ is a compact subset in $\mathcal{L}_k(N,A)$. 

\medskip

Let $Q=\Hom(T_k,N)$ and $Q_x=\Hom_{1^k\rightarrow x}(T_k,N)$. Recall that $T_k$ is the three-cube $\{-1,0,1\}^k$.  We have by construction 5 in chapter \ref{morpro} that if $v\neq 0$ then $t\circ\Psi_v$ is uniformly random in $C^k(N)$ as $t$ is randomly chosen from $Q_x$. Let $g_v:N\rightarrow\mathcal{L}(Q,A)$ denote the function which maps $x\in N$ to the restriction of $t\mapsto \rho(t\circ\Psi_v)$ to $Q_x$. By lemma \ref{csmtechnical} we obtain that $g_v$ is continuous for every $v\neq 0^k$. Let $$g=\sum_{v\in\{0,1\}^k\setminus 0^k}(-1)^{h(v)}g_v.$$ We have that $g:N\rightarrow\mathcal{L}(Q,A)$ is continuous.  We have by lemma \ref{threecsum} that for $x$ in $N$ the value of $g(x)$ is equal to the function $t\mapsto \rho_x(t\circ\omega)-\rho_x(t\circ\Psi_0)$ (defined on $Q_x$). Now let $Z=C^k(N)\times_N C^k(N)$ and $g':N\rightarrow\mathcal{L}(Z,A)$ be the function defined by $g'(x)=\mathcal{E}(\rho_x)$ (using the notation of lemma \ref{diffcont}). The continuity of $g$ and construction 4 in chapter \ref{morpro} imply that $g'$ is continuous and so $g'(N)$ is compact in $\mathcal{L}(Z,A)$. Using that $M=\mathcal{E}^{-1}(g'(N))$ and lemma \ref{diffcont} we obtain that $M$ is homeomorphic to a continuous $A$ bundle over $N$ and so it is compact.
\end{proof}

\medskip

Now we define cubes of dimension $k$ on the compact topological space $M$.
Let $f:\{0,1\}^k\rightarrow M$ be a function. We have for every $v\in\{0,1\}^k$ that $\rho_{\pi(f(v))}=f(v)+a(v)$ for some element $a(v)$ in $A$. We say that $f$ is in $C^k(M)$ if $\pi\circ f\in C^k(N)$ and
$$\sum_{v\in\{0,1\}^k}a(v)(-1)^{h(v)}=\rho(\pi\circ f).$$ 
Lemma \ref{threecsum} shows that this is equivalent with the requirement that
\begin{equation}\label{zecubesum}
\sum_{v\in\{0,1\}^k}f(v)(t\circ\Psi_v)(-1)^{h(v)}=0
\end{equation}
for some (and thus for every) $t\in\Hom_{\pi\circ f\circ\omega^{-1}}(T_k,N)$. (Recall that $f(v)$ is a $A$ valued function on $C^k_{\pi(f(v))}$.)
In general $f$ is in $C^n(M)$ if $\pi\circ f\in C^n(N)$ and every $k$-dimensional face restriction of $f$ is in $C^k(M)$.
It follows from (\ref{zecubesum}) that the nilspace structure on $M$ depends only on the set $M$ itself. On the other hand if $\rho$ and $\rho'$ differ by a coboundary of degree $k-1$ then they define the same set $M$. This means that the nilspace $M$ depends only on the element in $H_{k-1}(N,A)$ represented by $\rho$.
The next lemma is crucial.

\begin{lemma}\label{miscompact} $M$ is a compact nilspace.
\end{lemma}

\begin{proof}
Theorem \ref{bundec} shows that $M$ satisfies the nilspace axioms. Proposition \ref{meascont} defines a compact topology on $M$. By lemma \ref{kcompchar} it is enough to show that $C^k(M)$ is a closed subset of $M^{\{0,1\}^k}$.

We will use the probability space $\Hom(T_k,N)$ which can be looked at as a CSM with fibers $\Hom_{c\circ\omega^{-1}}(T_k,N)$ where $c\in C^k(N)$.
Let $Q^k(M)$ denote the set of all functions $f:\{0,1\}^k\rightarrow M$ such that $c=\pi\circ f \in C^k(N)$.
It is clear that $Q^k(M)$ is a closed subset of $M^{\{0,1\}^k}$. Now we define a map $\phi:Q^k(M)\rightarrow\mathcal{L}(\Hom(T_k,N),A)$ by
$$\phi(f)(g)=\sum_{v\in\{0,1\}^k}f(v)(g\circ\Psi_v)(-1)^{h(v)}$$
where $g\in\Hom_{c\circ\omega^{-1}}(T_k,N)$. This definition implies that $\phi(f)(g)$ is always a constant function. Let us denote this constant by $\phi_2(f)\in G$. Formula (\ref{zecubesum}) implies that $C^k(M)=\phi_2^{-1}(0)$.
Since the map $\phi$ is continuous we obtain that $C^k(M)$ is closed.
\end{proof}

\medskip

\begin{lemma} Let $M$ be a compact nilspace which is a degree-$k$ extension of a compact nilspace $N$ by a compact abelian group $A$. Then there is a measurable cross-section for this extension and therefore a measurable cocycle.
\end{lemma}

\begin{proof}
Let $\pi : M \to N$ be the projection of the extension. Consider the set $P = \{ (y,x) \in N \times M ~|~ \pi(x) = y \}$. A cross-section for $\pi$ is a subset of $P$ which happens to be the graph of a function $N \to M$. Corollary 18.7 of \cite{Ke} says that a sufficient condition for a Borel cross-section to exist is that for some Borel function $\mu : N \to \mathcal{P}(M)$, we have $\mu_y(P \cap (\{y\} \times M)) >0$. The measures $\mu_y$ from the CSM structure of $M$ satisfy this, and the map $y \to \mu_y$ is not only Borel but continuous by definition. 
\end{proof}

\begin{theorem}\label{mmorphiscont} Let $\phi:N\rightarrow M$ be a Borel measurable morphism between two finite step compact nilspaces. Then $\phi$ is continuous.
\end{theorem}

\begin{proof} Assume that $M$ is a $k$ step nilspace. We consider the CSM defined by $\psi_0:C^{k+1}(N)\rightarrow N$. For every $0\neq v\in\{0,1\}^{k+1}$ we define the function $f_v:N\rightarrow\mathcal{L}(C^{k+1}(N),M)$ so that $f_v(x)$ is the restriction of $\phi\circ\psi_v$ to the space $C_x^{k+1}(N)$.
Using the fact that $\psi_v$ is measure preserving and lemma \ref{csmtechnical} we obtain that $f_v$ is continuous for every $0\neq v\in\{0,1\}^{k+1}$. Let $f:N\rightarrow\mathcal{L}(C^{k+1}(N),M)$ be the function which maps $x\in N$ to the constant function with value $\phi(x)$ on $C_x^{k+1}(N)$. The compact nilspace structure on $M$ guarantees that $f$ depends continuously on the system $\{f_v\}_{0\neq v\in\{0,1\}^{k+1}}$ and thus $f$ is continuous. It follows that $\phi$ is continuous.  
\end{proof}

\begin{lemma} Let $N$ be a compact finite step nilspace. Let $K$ be a degree $k$-extension of $N$ by a compact abelian group $A$. Let $S:N\rightarrow K$ be measurable cross-section and $\rho$ be the associated measurable cocycle. Then $K$ is isomorphic as a compact nilspace to the extension $M$ constructed in proposition \ref{meascont}.
\end{lemma}

\begin{proof} The isomorphism is given by the map $\theta:K\to M$ defined by $\theta(x)=\rho_{\pi(x)}+(x-S(\pi(x)))$. By the definition of the cubic structure on $M$ the map $\theta$ is an isomorphism of abstract nilspaces.
Since both $K$ and $M$ are compact Hausdorff spaces if we show $\theta$ is continuous it will automatically be a homeomorphism.  It is clear that $\theta$ is measurable and so theorem \ref{mmorphiscont} finishes the proof.
\end{proof}

\begin{corollary}\label{classrep} The isomorphism class (as a topological nilspace) of every degree $k$ extension of $N$ by $A$ is represented by some element in $H_k(N,A)$ as described in proposition \ref{meascont} and lemma \ref{miscompact}.
\end{corollary}

\subsection{Finite rank nilspaces and averaging}

Let $N$ be a compact $k$-step nilspace. We have from theorem \ref{bundec} that $N$ is a degree $k$-bundle with structure groups $A_1,A_2,\dots,A_k$. The compactness of $N$ implies that the structure groups are compact abelian groups. We define the rank $\rk(N)$ by
$$\rk(N)=\sum_{i=1}^k \rk(\hat{A_i})$$
where $\hat{A_i}$ is the Pontrjagin dual of $A_i$ and $\rk(\hat{A_i})$
is the minimal number of generators of $\hat{A_i}$.

A result of Gleason \cite{Gl} implies that if a compact Lie group $G$
acts freely and continuously on a completely regular topological space
$X$, then the quotient map $X \to X/G$ is automatically a locally
trivial bundle. Because of this, we have that finite rank
nilspaces are iterated locally trivial fibrations of finite
dimensional compact abelian groups. Topologically, they are finite
dimensional manifolds.

Finite rank abelian groups are direct products of finite dimensional tori's and finite abelian groups.
There is a natural way of metrizing them. For two elements $x,y\in\mathbb{R}^n/\mathbb{Z}^n=\mathbb{T}_n$ we define their distance $d_2(x,y)$ as the minimal possible Euclidean distance between a preimage of $x$ and a preimage of $y$ under the map $\mathbb{R}^n\rightarrow\mathbb{R}^n/\mathbb{Z}^n$. If the abelian group is not connected then points in different connected components have infinite distance.

Let $X_1$ and $X_2$ be two Borel random variables taking values in a finite rank compact abelian group $A$.
In general there is no natural way of defining their expected values. However if they take values in small diameter sets in $A$ then there is a canonical way of defining their expected value and it will satisfy $\mathbb{E}(X_1+X_2)=\mathbb{E}(X_1)+\mathbb{E}(X_2)$.

Let $a\in \mathbb{T}^n$ be an element and $B_r(a)$ be the open ball of radius $r$ around $a$. Let $a'\in\mathbb{R}^n$ be an arbitrary preimage of $a$ under the homomorphism $\mathbb{R}^n\rightarrow\mathbb{T}_n$.
If a Borel random variable $X$ takes all its values in $B_{1/4}(a)$ then there is a unique way of lifting $X$ to a random variable $X'$ on $\mathbb{R}^n$ in a way that the values are closer than $1/4$ to $a'$. We define $\mathbb{E}(X)$ as the image of $\mathbb{E}(X')$ under the map $\mathbb{R}^n\rightarrow\mathbb{T}_n$.
It is easy to see that $\mathbb{E}(X)$ does not depend on the choice of $a$.
If $m$ random variables take their values in sets of diameter at most $1/5n$ then the additivity of the expected value is guaranteed.
The next lemma is an easy application of averaging.

\begin{lemma}\label{smallco} Let $N$ be an $l$-step nilspace and $A$ be a finite rank abelian group. Then there is an $\epsilon$ such that every Borel measurable cocycle $\sigma:C^k(N)\rightarrow A$ of degree $k-1$ with $d_2(\sigma(c),0)\leq\epsilon$ for every $c\in C^k(M)$ is a coboundary.
\end{lemma}

\begin{proof}
Let $m$ be an element in $N$. By lemma \ref{restbund} the set $$\Omega_m=\Hom_{0^k\mapsto m}(\{0,1\}^k,N)$$ is a sub-bundle in $N^{\{0,1\}^k}$ and so the Haar measure gives a probability space on $\Omega_m$. Let $c$ be a random element in $\Omega_m$. 
We define $g(m)=\mathbb{E}_{\Omega_m}(\sigma(c))$.
The expected value makes sense because $\sigma$ is always close to $0$.
We claim that $\sigma$ is a coboundary corresponding to the function $g:N\rightarrow A$.

Let $f\in C^n(N)$ be an arbitrary element.
Let $T_n$ be the three-cube. For an arbitrary function $f:\{0,1\}^n\rightarrow N$ in $C^n(N)$ we define the probability space $\Omega=\Hom_{f\circ\omega^{-1}}(T_n,N)$.
Let $c$ be a random element in $\Omega$ (see construction 2 in chapter \ref{morpro}). We have by lemma \ref{threecsum} that
$$\sigma(f)=\sum_{v\in\{0,1\}^n}\sigma(c\circ\Psi_v)(-1)^{h(v)}.$$
According to construction 4 in chapter \ref{morpro}, the distribution of $c\circ\Psi_v$ is given by $\Omega_{f(v)}$.
This means by taking the expected value of both sides that
$$\sigma(f)=\sum_{v\in\{0,1\}^n}g(f(v))(-1)^{h(v)}.$$
\end{proof}

\subsection{The number of finite rank nilspaces}

The main result of this chapter is the following.

\begin{theorem}\label{countably} There are countably many finite rank compact nilspaces up to isomorphism.
\end{theorem}

Recall that a nilspace is of finite rank if the dual groups of its structure groups are all finitely generated. A compact nilspace is of finite rank if and only if it is finite dimensional. It is clear that there are at least countably many finite rank nilspaces so it remains to show the upper bound.
We will need the following two lemmas.

\begin{lemma}\label{almostzero} Let $A$ be a finite rank compact abelian group and $N$ be a compact $n$-step nilspace. Let $\rho:C^k(N)\rightarrow A$ be a Borel measurable cocycle such that $\rho=0$ for almost every element in $C^k(N)$. Then $\rho$ is a coboundary.
\end{lemma}

\begin{proof} We use the notation from Proposition \ref{meascont}. Let $S\subseteq N$ be the set of points $x$ such that $\rho_x=0$ almost surely on $C_x^k(N)$. It is clear that the measure of $S$ is $1$. For every $x$ in $S$ the set $\{\rho_x+a~|~a \in A\} \subset\mathcal{L}_k(N,A)$ is the set of constant functions. Usin proposition \ref{meascont} and the fact that $S$ is dense in $N$ we obtain that for every $x$ the function $\rho_x$ is almost surely equal to a constant $f(x)$.
By subtracting the coboundary corresponding to $f$ from $\rho$ we get a cocycle $\rho'$ with the property that $\rho'_x=0$ for every $x\in N$ almost surely on $C_x^k(N)$.   We claim that $\rho'=0$ everywhere. 
Let $c\in C^k(N)$ be an arbitrary cube and let $f=c\circ\omega^{-1}$ be defined on $X\subset T_n$ as in construction 2 from chapter \ref{morpro}. Let $t$ be a random element from $\Hom_f(T_n,N)$. Then by lemma \ref{threecsum} we have that $\rho'(c)=\beta(t,\rho')$. This together with the fact that $\beta(t,\rho')=0$ shows the claim.
\end{proof} 

\begin{lemma} \label{smallco2} Let $N$ be $l$-step nilspace and $A$ be a finite rank abelian group. Then there is an $\epsilon$ such that every Borel measurable cocycle $\sigma:C^k(N)\rightarrow A$ of degree $k-1$ with $d_2(\sigma(c),0)\leq\epsilon$ for almost every $c\in C^k(M)$ is a coboundary.
\end{lemma}

\begin{proof} As the statement is very similar, the proof is also very similar to the proof of lemma \ref{smallco}. By using the notation from the proof of lemma \ref{smallco} we only explain the difference between the proofs. Let $S$ be the set of elements $m\in N$ for which $\sigma$ restricted to $C_m^k(N)$ is almost surely $\epsilon$ close to $0$. By the condition of the lemma the set $S$ has measure $1$. Inside $S$ we can define the function $g$ as in the proof of lemma \ref{smallco}. We set $g(m)=0$ if $m\neq S$. The same argument shows that the coboundary $\sigma_2$ corresponding to $g$ is almost everywhere equal to $\sigma$. By lemma \ref{almostzero} the difference $\sigma-\sigma_2$ is a coboundary. This implies that $\sigma$ is a coboundary.
\end{proof}

\medskip

\noindent{\it Proof of theorem \ref{countably}:}~~By induction on $k$ we prove that there are countably many finite rank $k$-step nilspaces. The statement is trivial for $k=0$. Assume that it is true for $k-1$. It is enough to show that for every finite rank $k-1$ step nilspace $N$ and finite rank abelian group $A$ there are at most countably many non isomorphic degree $k$ extensions of $N$ by $A$.  By contradiction assume that $\mathcal{S}$ is an uncountable family of non isomorphic such extensions.

Since $N$ is finite dimensional we can find a finite system $U_1,U_2,\dots,U_r$ of disjoint open sets each homeomorphic to a (possibly zero dimensional) Euclidean space such that $U=\cup_{i=1}^r U_i$ has measure $1$ in $N$.
Let $M$ be a degree $k$-extension of $N$ by $A$. Since every $A$ fibration of a Euclidean space is trivial, we can choose a Borel representative system $x:N\rightarrow M$ for the $A$ fibers over $N$ such that $x$ restricted to each $U_i$ is continuous. Let $T\subseteq C^{k+1}(N)$ denote the set of cubes whose vertices are all in $U$. We can decompose $T=\cup_{i=1}^l T_i$ according to the index sets of $U_i$'s containing the vertices. The cocycle $\rho:C^{k+1}(N)\rightarrow A$ on $N$ which is given by $x$ is clearly continuous on each set $T_i\subset T$. Since there is a separable $L^\infty$ dense system of continuous functions of the form $T_i\rightarrow A$ on each $T_i$ we get that in the uncountable family $\mathcal{S}$ there are two extensions $M_1$ and $M_2$ such that the above defined cocycles $\rho_1$ and $\rho_2$ are at most $\epsilon>0$ close (for an arbitrary $\epsilon$) on $T$. Then lemma \ref{smallco2} shows that $\rho_1-\rho_2$ is a coboundary which is a contradiction by the results in chapter \ref{chapmeas}.

\subsection{The Inverse limit theorem}

In this chapter we develop an iterative method of finding finite rank nilspace factors of compact nilspaces. Using this, our main result will establish compact nilspaces as inverse limits of finite rank ones. As a preparation we start with 
a definition and a lemma which show how to produce fiber surjective factors from apropriate cross sections in $k$-step nilspaces.

\begin{definition}\label{crossfactdef} Let $N$ be a $k$-step compact
  nilspace. Let $M=\mathcal{F}_{k-1}(N)$ with projection
  $\pi:N\rightarrow M$ and let $A$ be the $k$-th structure group of
  $N$. Let $\psi:M\rightarrow M_2$ be a fiber surjective factor of
  $M$. Assume furthermore that $S:M\rightarrow N$ is a measurable
  cross section such that the corresponding cocycle
  $\rho:C^{k+1}(M)\rightarrow A$ satisfies
  $\rho(c_1)=\rho(c_2)$ for every pair $c_1,c_2\in C^{k+1}(M)$
  with $\psi\circ c_1=\psi\circ c_2$. Then we say that the cross
  section $S$ is consistent with the factor $M_2$ of $M$.
\end{definition}

\begin{lemma}\label{crossfactor} Let us use the notation and assumptions from definition \ref{crossfactdef}. Let $\sim$ be the equivalence relation on $N$ defined by $x\sim y$ if and only if $\psi(\pi(x))=\psi(\pi(y))$ and $x-S(\pi(x))=y-S(\pi(y))$.
Then $\sim$ defines a fiber surjective factor $N_2$ of $N$ which an extension of $M_2$ by $A$. 
\end{lemma}

\begin{proof} Let $\rho':C^{k+1}(M_2)\rightarrow A$ be the well defined cocycle computed on a cube $c$ as the value of $\rho$ on the preimage of $c$ under $\psi$.  Let $N_2$ be the nilspace obtained from $M_2$ by the extension defined by $\rho'$. The elements of $M_2$ can be represented as in proposition \ref{meascont}. Recall that elements of $M_2$ are shifts of restricted cocycles of type $\rho'_x$. Let $f:N\rightarrow N_2$ be the map defied by $f(x)=x-S(\pi(x))+\rho'_{\psi(\pi(x))}$. It is clear from the definitions that $f$ is a fiber surjective morphism and that $f(x_1)=f(x_2)$ if and only if $x_1\sim x_2$. Furthermore by theorem \ref{mmorphiscont} we obtain that $f$ is a (continuous) isomorphism between compact nilspaces.
\end{proof}

\begin{theorem}[Inverse limit theorem]\label{invlim} Every $k$-step compact nilspace is an inverse limit of finite rank nilspaces. The maps used in the inverse system are all fiber surjective morphisms.
\end{theorem}

\begin{proof}

We prove the theorem by induction on $k$. If $k=0$ then there is nothing to prove.
Assume that it is true for $k-1$.
Let $N$ be a $k$-step nilspace with structure groups $A_1,A_2,\dots,A_k$.
By induction $M=\mathcal{F}_{k-1}(N)$ is the inverse limit of a system $M_1\leftarrow M_2\leftarrow\dots$ where the maps are all fiber surjective morphisms. Let us denote by $\tau_i$ the projection to $M_i$ and let $\pi$ be the projection $N\rightarrow M$. Let $\mathcal{Q}_i$ denote the collection of open sets of the form $\tau_i^{-1}(U)$ where $U$ is open in $M_i$. Since $M$ is a compact Hausdorff space, its topology is generated by the system $\{\mathcal{Q}_i\}_{i=1}^\infty$.

Since $A_k$ is a compact abelian group we have that $A_k$ is the inverse limit of finite rank compact abelian groups. This implies that there is a descending chain $A_k=B_0>B_1>\dots$ of closed subgroups with trivial intersection such that each factor $A_k/B_i$ is of finite rank. The nilspace $N$ is the inverse limit of the nilspaces $N/B_i$ and all the maps $N\rightarrow N/B_i$ are fiber surjective.

Our goal is to create a sequence of fiber surjective maps $\psi_i:N/B_i\rightarrow N_i$ and $\psi_i':N_i\rightarrow N_{i-1}$ for some finite rank nilspaces $\{N_i\}_{i=0}^\infty$ such that 
\begin{enumerate}
\item For every $i$, the restriction of $\psi_i$ to each $\sim_{k-1}$ class of $N/B_i$ is injective. In other words, the $k$-th structure group of $N_i$ is $A_k/B_i$. 
\item $\psi'_i\circ\psi_i=\psi_{i-1}$ holds for every $i=1,2,\dots$.
\item There is a strictly increasing sequence of natural numbers $\{h_i\}_{i=1}^\infty$ such that $\mathcal{F}_{k-1}(N_i)\simeq M_{h_i}$ and $\psi_i$ composed with the projection to $\mathcal{F}_{k-1}(N_i)$ is equal to the composition of the projection $N/B_i\rightarrow M$ and $\tau_{h_i}:M\rightarrow M_{h_i}$.
\end{enumerate} 
By abusing the notation we will denote the projection from $N_i$ to $M_{h_i}$ by $\pi$.
It is clear from these properties that every two points in $N$ are separated by a map $\psi_i$ for some number $i$. This implies that $N$ is the inverse limit of the system $\{\psi'_i\}_{i=1}^\infty$.
We construct this sequence recursively.

Define $N_0$ as the one point nilspace. Assume that
$\{N_i,\psi_i,\psi'_i,h_i\}_{i=1}^m$ are already constructed.  Let
$\{W_i\}_{i=1}^r$ be a system of closed subsets in $M_{h_m}$ such that
$\pi^{-1}(W_i)\subset N_m$ is a trivial $A_k/B_m$ bundle over $W_i$
for every $1\leq i\leq r$.  The existence of $\{W_i\}_{i=1}^r$ follows
easily from compactness and Gleason's previously mentioned result
\cite{Gl} that implies that any free action of a compact Lie group on
a compact space is automatically locally trivial. For each $1\leq
i\leq r$ let $\theta_i:W_i\rightarrow N_m$ be a continuous cross
section (which means that $\theta_i$ is continuous and its composition
by $\pi$ is the identity map of $W_i$).

For every $1\leq a\leq r$ let $W'_a$ denote the preimage of the set $\{\theta_i(x):x\in W_a\}$ under the composition of $N/B_{m+1}\rightarrow N/B_m$ and $\psi_m$. It is clear that $W'_a$ is a $B_{m+1}/B_m$ bundle over $\tau_{h_m}^{-1}(W_a)\subset M$. 
From Gleason's result we have that $W_a'$ as a $B_m/B_{m+1}$-bundle is locally trivial. Let $d$ be a metrization of $N/B_{m+1}$. For an arbitrary epsilon and every point $p\in \tau_{h_m}^{-1}(W_a)$ we can choose an open neighborhood $U_p$ of $p$ with the following three properties.

\begin{enumerate}
\item  there is a continuous cross section $S_p:U_p\rightarrow N/{B_{m+1}}$ above $U_p$
\item $S_p(U_p)$ has diameter at most $\epsilon$
\item $U_p\in\mathcal{Q}_{t(p)}$ for some $t(p)\in\mathbb{N}$.
\end{enumerate}

It is clear that we can guarantee the first two properties. The last property follows from the fact that the topology on $M$ is generated by the topologies on $M_i$.
 The compactness of $M$ implies that there are finitely many points $p_1,p_2,\dots,p_n$ such that $\{U_{p_i}\}_{i=1}^n$ is a covering system of $\tau_{h_m}^{-1}(W_a)$. Let $t_a=\max\{t(p_i)\}_{i=1}^n$. We have that every set in $\{U_{p_i}\}_{i=1}^n$ is in $\mathcal{Q}_{t_a}$.

 Now we can create a Borel measurable cross section $S_a:\tau_{h_m}^{-1}(W_a)\rightarrow N/B_{m+1}$ with the following properties.

 \begin{enumerate}
 \item $S_a$ is continuous on every preimage $\tau_{t_a}^{-1}(v)$ where $v\in M_{t_a}$
 \item The diameter of $S_a(\tau_{t_a}^{-1}(v))$ is at most $\epsilon$ for every $v\in M_{t_a}$.
 \end{enumerate}

 This can be constructed by dividing $\tau_{h_m}^{-1}(W_a)$ into the atoms of the Boolean algebra generated by $\{U_{p_i}\}_{i=1}^n$ and then using one type of cross section for each atom.
Let $t$ be tha maximum of the numbers $\{t_a\}_{a=1}^r$ and $h_m+1$. Using these partial cross section we create a global cross section $S:M\rightarrow N/B_{m+1}$ with the following properties.
 
 \begin{enumerate}
 \item $S$ is continuous on every preimage $\tau_t^{-1}(v)$ where $v\in M_t$
 \item The diameter of $S(\tau_t^{-1}(v))$ is at most $\epsilon$ for every $v\in M_t$.
 \item If $\tau_{h_m}(v_1)=\tau_{h_m}(v_2)$ for some $v_1,v_2\in M$ then $\psi_m(S(v_1))=\psi_m(S(v_2))$ 
 \end{enumerate}
 The last property expresses the fact that $S$ modulo $B_m$ is the pre image of a cross section in $N_m$.
 The cross section $S$ can again be constructed by dividing $M$ into the atoms of the Boolean algebra generated by the sets $\tau_{h_m}^{-1}(W_a)$ and then using one type of cross section for each atom.

We denote by $\rho:C^{k+1}(M)\rightarrow A_k/B_{m+1}$ the cocycle given by $S$ on $M$.
If $\epsilon$ is small enough than we can guarantee that for any two
cubes $c_1,c_2\in C^{k+1}(M)$ with $c_1 \circ \tau_t = c_2 \circ \tau_t$ we have
\begin{equation}\label{rhoclose}
d_2(\rho(c_1)-\rho(c_2))\leq\epsilon_2.
\end{equation}
Furthermore by the third property of $S$ we have that if $c_1,c_2\in
C^{k+1}(M)$ satisfy $\tau_{h_m} \circ c_1 = \tau_{h_m}\circ c_2$ then $\rho(c_1)$ is congruent with $\rho(c_2)$ modulo $B_m$.

We have by lemma \ref{collection} that the map
$\beta_i:\Hom(\{0,1\}^i,M)\rightarrow\Hom(\{0,1\}^i,M_t)$ induced by $\tau_t$ is totally surjective and preimages of elements in $\Hom(\{0,1\}^i,M_t)$ are $(k-1)$-fold sub-bundles of $\Hom(\{0,1\}^i,M)$. See also construction 1 in chapter \ref{morpro}.
We define the function $\rho':C^{k+1}(M)\rightarrow A_k/B_{m+1}$ by
$$\rho'(c)=\mathbb{E}_{c'\in\beta_{k+1}^{-1}(\beta_{k+1}(c))}(\rho(c')).$$
It makes sense to use the expected value because (\ref{rhoclose}) implies that $\{\rho(c')|c'\in\beta_{k+1}^{-1}(\beta_{k+1}(c))\}$ has small diameter if $\epsilon_2$ is small enough. Note that we take the expected value according to the Haar measure on $\beta_{k+1}^{-1}(\beta_{k+1}(c))$.
Since $t>h_m$ we have that $\rho'$ is congruent with $\rho$ modulo $B_m$.

\medskip

We claim that $\rho'$ is a cocycle on $M$. This follows basically from the fact that the cocycle axioms are all linear equations on cubes of dimension $k+1$ and $k+2$ and expected value is additive. To be more precise we need to check the axioms for automorphisms and concatenations of cubes. Since a set of the form $\beta_{k+1}^{-1}(\beta_{k+1}(c))$ is automorphism invariant the first axiom trivially holds for $\rho'$. To see the second axiom we take two concatenated cubes $c_1,c_2\in C^{k+1}(M)$ with concatenation $c_3$ and embed them into a $k+2$ dimensional cube $c\in C^{k+2}(M)$ as restrictions to two adjacent $k+1$-dimensional faces $F_1$ and $F_2$ in $\{0,1\}^{k+2}$.
The concatenation of $F_1$ and $F_2$ is a diagonal subcube $F_3$. We have that the concatenation of $c_1$ and $c_2$ is the restriction of $c$ to $F_3$. Note that the existence of $c$ is guaranteed by lemma \ref{simpglue}.
We use the third point of lemma \ref{collection} to view $\Omega=\beta_{k+2}^{-1}(\beta_{k+2}(c))$ as a probability space.
As the fourth point in lemma \ref{collection} says, the probability spaces $\beta_{k+1}^{-1}(\beta_{k+1}(c_i))$, $i=1,2,3$ are faithfully embedded as factors and coupled in the big probability space $\Omega$. Using the concatenation property for $\rho$ in $\Omega$ (when the random cube is restricted to $F_1,F_2,F_3$) and linearity of expectation we obtain that $\rho'(c_3)=\rho'(c_1)+\rho'(c_2)$.

\medskip

Now we have that $\rho'$ is a cocycle and so $\rho''=\rho'-\rho$ is also a cocycle.
Note that $\rho''$ takes values in $B_m/B_{m+1}$.
We have by (\ref{rhoclose}) that $d_2(\rho''(c),0)\leq\epsilon_2$ holds for every $c\in C^{k+1}(M)$.
By lemma \ref{smallco} we get that $\rho''$ is a coboundary.

 Since the difference of $\rho$ and $\rho'$ is a coboundary corresponding to a function $g:M\rightarrow B_m/B_{m+1}$ we have that by adding $g$ to our cross section $S$ we get a new cross section $S'$ such that the cocycle corresponding to $S'$ is equal to $\rho'$. This means that on $N/B_{m+1}$ the cross section $S'$ is consistent with the factor $M_t$.
 The way we produced $g$ and $S'$ (see the proof of lemma \ref{smallco}) guarantees that it is continuous on the preimages of points in $M_t$ under $\tau_t$. 
 Using lemma \ref{crossfactor} we obtain that $S'$ defines a factor $N_{m+1}$ of $N/_{B_{m+1}}$. Furthermore we also have that $S'$ is congruent to $S$ modulo $B_m$. This implies that $N_m$ factors through $N_{m+1}$. 
\end{proof}

\subsection{Rigidity of morphisms}

Let $N$ and $M$ be compact $k$-step nilspaces and let $d$ be a metric
on $M$ (metrizing its topology). We say that a map $\phi:N\rightarrow
M$ is an $\epsilon$-almost morphism if for an arbitrary $c\in
C^{k+1}(N)$ there is $c'\in C^{k+1}(N)$ such that $d(\phi\circ c,c')\leq\epsilon$ point wise.

An $\epsilon$ modification of a map $\phi:N\rightarrow M$ is another map $\phi'$ satisfying $d(\phi(x),\phi'(x))\leq\epsilon$ for every $x\in N$.

\begin{theorem} For every finite rank $k$-step nilspace $M$ with metric $d$ there is a function $f:\mathbb{R}^+\rightarrow\mathbb{R}^+$ with $\lim_{x\to 0}f(x)=0$ and $\epsilon_0>0$ such that if $\phi:N\rightarrow M$ is a Borel $\epsilon$-almost morphism with $\epsilon<\epsilon_0$ from a compact $k$-step nilspace $N$ to $M$ then it can be $f(\epsilon)$-modified to a (continuous) morphism $\phi'$.
\end{theorem}

In the rest of this chapter we prove this theorem.

We go by induction on $k$. For $k=0$ there is nothing to prove. Assume that we have the statement for $k-1$.
The metric $d$ induces another metric $d'$ on $\mathcal{F}_{k-1}(M)$ such that $$d'(x',y')=\min\{d(x,y)|x,y\in M,~\pi_{k-1}(x)=x',\pi_{k-1}(y)=y'\}.$$

The assumption that $\phi$ is an $\epsilon$-morphism trivially implies that $\phi'=\pi_{k-1}\circ\phi$ is an $\epsilon$-morphism into $\mathcal{F}_{k-1}(M)$. By induction we can $f'(\epsilon)$-modify $\phi'$ to get a morphism $\phi_2:N\rightarrow\mathcal{F}_{k-1}(M)$.

We claim that there is a Borel measurable lift $\phi_3:N\rightarrow M$
of $\phi_2$ (where lift means that $\pi_{k-1}\circ\phi_3=\phi_2$) 
such that $d(\phi_3(x),\phi(x))\leq f'(\epsilon)+\epsilon$ and in particular
$\phi_3$ is an $\epsilon_2=f'(\epsilon)+2\epsilon$ almost-morphism.
To see this let $$G=\{(x,y):\pi_{k-1}(y)=\phi_2(x)\}\subset N\times M.$$
It is clear that $G$ is a compact set and it is an $A_k$ bundle over $N$ where $A_k$ is the $k$-th structure group of $M$. 
Let $p\in N$ be an arbitrary point. By Gleason's automatic local
triviality result \cite{Gl}, there is an open neighborhood $U_p\subset N$ of $p$ such that the $A_k$ bundle over $U_p$ in $G$ is trivial. This means that there is a continuous function $\tau:U_p\rightarrow N$ such that $\pi_{k-1}(\tau(x))=\phi_2(x)$ for every $x\in U_p$. By the definition of the $d'$ metric there is a constant $c\in A_k$ depending on $p$ such that $d(\phi(p),\tau(p)+c)<f'(\epsilon)+\epsilon$.
Since $\tau+c$ is a continuous function we have on some neighborhood $U'_p\subset U_p$ of $p$ that 
$d(\phi(x),\tau(x)+c)<f'(\epsilon)+\epsilon$ holds for every $x\in U'_p$. Let $f_p=\tau(x)+c$. The function $f_p$ is a local lift of $\psi_2$ satisfying our requirement. To finish the proof of the claim we choose a finite covering system of $N$ by sets of the form $U'_{p_i}$ and then on each atom of the finite boolean algebra generated by them we choose one function $f_{p_j}$ which is defined on it. The union of these functions is Borel measurable and satisfies our requirement.

Now we introduce an averaging process to get a function $\phi_4$ in the following way.
Let $P_2=\{0,1\}^{k+1}\setminus\{0^{k+1}\}$ be the corner of the $k+1$ dimensional cube $P$.
Using corollary \ref{sim2cor} and the fact that $\phi_2$ is a morphism we get that $\phi_3$ takes $k$-dimensional cubes in $N$ into $k$-dimensional cubes in $M$. This means that for every morphism $\gamma:P\rightarrow N$ the composition $\phi_3\circ\gamma|_{P_2}$ is a morphism of the corner $P_2$. For a morphism $\gamma:P\rightarrow N$ We denote by $Q(\gamma)\in M$ the unique completion of $\phi_3\circ\gamma|_{P_2}$ in $M$.

Now we define
$$\phi_4(x)=\mathbb{E}_{\gamma\in\Hom_f(P,N)}(Q(\gamma))$$
where $f$ maps the point $0^{k+1}$ to $x$.
From the fact that $\phi_2$ is a morphism we get that for every $\gamma\in\Hom_f(P,N)$ the element $Q(\gamma)$ is in the $\sim_{k-1}$ class of $\phi_3(x)$. This class is an affine copy of the $k$-th structure group $A_k$ of $M$. To show that the above averaging makes sense, we need to show that if $\epsilon_2$ is small enough than the values of $Q(\gamma)$ are in a small neighborhood in $A_k$. In fact we will show that $Q(\gamma)$ is close to $\phi_3(x)$ which will be also important later.
Using that $\phi_3$ is an almost morphism we get that $\phi_3\circ\gamma$ is close to an element in $C^{k+1}(N)$. The continuity of the cube structure shows that $\phi_3\circ\gamma$ is also close to a cube $q$ whose composition with $\pi_{k-1}$ is equal to $\pi_{k-1}(\gamma(P))$.
Let us write $\phi_3\circ\gamma$ as $q+r$ where $r:P\rightarrow A_k$ is some map. The fact that $\phi_3\circ\gamma$ is an almost cube translates to the fact that $$\sum_{v\in\{0,1\}^{k+1}}r(v)(-1)^{h(v)}=z$$ is close to $0$. On the other hand theorem \ref{bundec} shows that $Q(\gamma)=\phi_3(x)-z$.

\medskip

The next step is to show that $\phi_4$ is cube preserving.
According to lemma \ref{cubechar} we need to show that $k+1$ dimensional cubes map to cubes under $\phi_4$.
Let $T_{k+1}$ be the $3$-cube with subset $X$ as in construction 2 in chapter \ref{morpro}.
Let $B=T_{k+1}\setminus X$.
By abusing the notation a cube $c\in C^{k+1}(N)$ can be interpreted as a function $c:X\rightarrow N$.
For every element $\kappa\in\Hom_c(T_{k+1},N)$ we denote by
$Q(\kappa)\in C^{k+1}(M)$ the cube obtained by first taking the unique
extension of $\phi_3 \circ \kappa|_B$ to a morphism $T_{k+1}\rightarrow M$ and then restricting it to $X$.
Now
$$c_2=\mathbb{E}_{\kappa\in\Hom_c(T_{k+1},N)}Q(\kappa)$$
makes sense if $\epsilon_4$ is small enough and by theorem \ref{bundec} it will be a cube.
On the other hand  By construction 4 in chapter \ref{morpro} we obtain
that $c_2=\phi_4\circ c$ evaluated at the point $1^{k+1}
\in X$. The symmetries of $T_n$ guarantee that $c_2=\phi_4\circ c$.

\medskip

The last step is to show that $\phi_4$ is continuous. This is rooted
in the type of averaging which produces $\phi_4$.  We use a similar
argument as in chapter \ref{chapmeas}. The probability spaces
$C^{k+1}_x(N)$ are forming a CSM in $C^{k+1}(N)$.  Let
$\psi_v:C^{k+1}(N)\rightarrow N$ be the coordinate function defined by
$\psi(f)=f(v)$.  Let $\psi_{v,x}$ be the restriction of $\psi_v$ to
$C_x^{k+1}(N)$. The fact that $N$ is a topological nilspace shows that
for fixed $v$, the function $N \to \mathcal{L}(C^{k+1}(N),N)$ given by
$x \mapsto \psi_{v,x}$ is continuous with respect to the topology on
$\mathcal{L}(C^{k+1}(N),N)$ defined in section \ref{contbundec}. More
generally, it is not hard to see that, if $\alpha$ is an arbitrary
measurable function from $N$ to some Borel space $T$, then the
function $N \to \mathcal{L}(C^{k+1}(N),T)$ given by $x \mapsto
\alpha \circ \psi_{v,x}$ is continuous. (It is obviously enough to show
this for functions $\alpha:N\rightarrow\{0,1\}$ which are
characteristic functions of Borel sets in $N$, which is easy to see.)
In particular we have that $\phi_3\circ\psi_{v,x}$ depends
continuously on $x$. The average defining $\phi_4(x)$ is on the space
$C_x^{k+1}(N)$ using the functions $\phi_3\circ\psi_{v,x}$ in a
continuous way. This shows that claim.

\subsection{Nilspaces as nilmanifolds}

Let $N$ be a compact $k$-step nilspace. By abusing the notation we denote by $\aff(N)$ the set of translations which are continuous  functions from $N$ to $N$.
Note that continuous translations are fiber surjective automorphisms of $N$ and so they are measure preserving.

Let $d$ be a metrization of the topology on $N$. This induces a metric $t$ on $\aff(N)$ defined by
\begin{equation}\label{poldist}
t(g,h)=\max_{x\in N}d(g(x),h(x)).
\end{equation}
It is easy to see that $\aff(N)$ is a Polish group with this metrization.
Similarly we will denote by $\aff_i(N)$ the set of continuous translations of height $i$.
Note that by theorem \ref{mmorphiscont} the set of Borel translations of height $i$ is the same as the set of continuous translations of height $i$.

From now on we assume that $N$ is a finite rank $k$-step nilspace. Our goal is to show that $\aff(N)$ is a $k$-nilpotent Lie group which acts transitively on the connected components of $N$. We will also show that if $N$ is connected then it is a nilmanifold obtained from $\aff(N)$ whose nilspace structure is given by the filtration $\{\aff_i(N)\}_{i=1}^k$.

From lemma \ref{loctrans2} we obtain that $\sim_{k-1}$ classes are imprimitivity domains of $\aff(N)$.  This means that the action on $\sim_{k-1}$ classes induces a homomorphism $h:\aff(N)\rightarrow\aff(\mathcal{F}_{k-1}(N))$. It is clear that $h(\aff_i(N))\subseteq\aff_i(\mathcal{F}_{k-1}(N))$.
Let $M=\mathcal{F}_{k-1}(N)$ and let $d'$ be the metric on $M$ defined as the Hausdorff distance $d'(x,y)=d(\pi_{k-1}^{-1}(x),\pi_{k-1}^{-1}(y))$. Let us denote by $t'$ the metric on $\aff(M)$ defined similarly as in (\ref{poldist}) from $d'$.

\begin{lemma}\label{smalltrans} Let $i$ be a natural number. There is a positive number $\epsilon>0$ such that if $\alpha\in\aff_i(M)$ satisfies $t'(\alpha,1)\leq\epsilon$ then there is $\beta\in\aff_i(N)$ with $h(\beta)=\alpha$.
\end{lemma}

\begin{proof}  The translation bundle $\mathcal{T}^*=\mathcal{T}^*(\alpha,N,i)$ is a $k-i$ degree extension of $M$ by $A_k$. Our goal is to show if $\epsilon$ is small enough then the cocycle describing the extension is a coboundary.
If $\epsilon$ is small enough then we can choose a Borel representative system $S$ for the fibers of the map $\mathcal{T}^*\rightarrow M$ such that $(x,y)\in\mathcal{T}$ represents an element in $S$ then $d(x,y)\leq\epsilon_2$. A standard compactness argument shows that if $\epsilon_2$ is small enough then the cocycle corresponding to $S$ is also small.  Then lemma \ref{smallco} and lemma \ref{transext} finish the proof.

\end{proof}

\begin{lemma}\label{kerchar1} Assume that $i>k$. Then $${\rm ker}(h)\cap\aff_i(N)=\hom(M,\mathcal{D}_{k-i}(A_k)).$$
\end{lemma}

\begin{proof} The elements of ${\rm ker}(h)$ are those translations which stabilize every $\sim_{k-1}$ class in $N$. It follows that if $\alpha\in{\rm ker}(h)$ then the map $\alpha':x\mapsto \alpha(x)-x$ can be viewed as a map from $M$ to $A_k$.
Lemma \ref{transchar2} implies that $\alpha'$ arises this way if it is a homomorphism of $M$ to $\mathcal{D}_{k-1}(A_k)$. It is easy to see that if in addition $\alpha'\in\aff_i(N)$ then it is a morphism to $\mathcal{D}_{k-i}(A_k)$.
\end{proof}

\begin{lemma}\label{variation1} Let $k,r\geq 1$ be two natural numbers and $A,B$ two compact abelian groups. Assume that $B$ is finite dimensional. Then there is a constant $\epsilon=\epsilon(r,B)>0$ such that if $\phi\in\Hom(\mathcal{D}_k(A),\mathcal{D}_r(B))$ satisfies $d(\phi(x),\phi(y))\leq\epsilon$ for every $x,y\in A$ then $\phi$ is a constant function.
\end{lemma}

\begin{proof} Using that $\Hom(\mathcal{D}_k(A),\mathcal{D}_r(B))\subseteq\Hom(\mathcal{D}_1(A),\mathcal{D}_r(B))$ we can assume that $k=1$. Let $\phi$ be an arbitrary non-constant morphism from $\mathcal{D}_1(A)$ to $\mathcal{D}_r(B)$.

For any $t\in A$ and function $f:A\rightarrow B$ we denote by $\Delta_t f$ the function $x\rightarrow f(x)-f(x+t)$. With this notation we have that if $f\in\Hom(\mathcal{D}_1(A),\mathcal{D}_i(B))$ then
$\Delta_t f\in\Hom(\mathcal{D}_1(A),\mathcal{D}_{i-1}(B)).$
for every $t\in A$.
It follows that $\Delta_{t_1,t_2,\dots,t_r}\phi$ is constant for every $r$-tuple of elements $t_1,t_2,\dots,t_r$ in $A$.
We obtain that there is a number $i<r$ and elements $t_1,t_2,\dots,t_i\in A$ such that $\phi'=\Delta_{t_1,t_2,\dots,t_i}\phi$ is non-constant but $\Delta_t \phi'$ is constant for every $t\in A$.
It follows that $\phi'$ is a non-constant affine group homomorphism from $A$ to $B$.
In particular there is a constant $c$ depending only on $B$ such that there are $x,y\in A$ with $d(\phi'(x),\phi'(y))\geq c$.
We get that if the variation $\max_{x,y}d(\phi(x),\phi(y))$ is too small this is impossible.
In other words there is a non-zero lower bound (depending only on $B$ and $r$) for the variation of $\phi$.
\end{proof}

\begin{corollary}\label{variation2} Let $r\geq 1$ be a natural number and $B$ a compact finite dimensional abelian group. Let $N$ be a $k$-step compact nilspace. Then there is a constant $\epsilon=\epsilon(r,B)>0$ such that if $\phi\in\Hom(N,\mathcal{D}_r(B))$ satisfies $d(\phi(x),\phi(y))\leq\epsilon$ for every $x,y\in N$ then $\phi$ is a constant function.
\end{corollary}

\begin{proof} Assume that $d(\phi(x),\phi(y))<\epsilon$ for every $x,y\in N$ where $\epsilon=\epsilon(r,B)$ is the constant from lemma \ref{variation1}. We prove by induction on $k$ that $\phi$ is constant.

If $k=1$ then $N$ is abelian and lemma \ref{variation1} finishes the proof.
Assume that the statement holds for $k-1$. We get from lemma \ref{variation1} that $\phi$ is constant on the $\sim_{k-1}$ classes of $N$. This means that $\phi$ can be regarded as a function on $\mathcal{F}_{k-1}(N)$. Then our assumption finishes the proof.
\end{proof}

\begin{lemma}\label{kerchar2} The group ${\rm ker}(h)$ is a Lie group.
\end{lemma}

\begin{proof} Let $x\in N$ be an arbitrary element and let $F$ be the stabilizer of $x$ in ${\rm ker}(h)$. Then by lemma \ref{kerchar1} we obtain that ${\rm ker}(h)=F\times A_k$.
It follows from corollary \ref{variation2} that $F$ is discrete and since $A_k$ is a Lie-group the proof is complete.
\end{proof}

If $L$ is a topological group then we denote the connected component of the unit element by $L^0$. 

\begin{theorem}\label{transurj} Let $i$ be a natural number. Then the following statements hold.
\begin{enumerate}
\item $\aff_i(N)$ and $\aff_i(N)^0$ are Lie groups,
\item $h(\aff_i(N)^0)=\aff_i(M)^0$.
\end{enumerate}
\end{theorem}

\begin{proof}  We prove the statements by induction on $k$. If $k=1$ then $N$ is an abelian Lie-group and all statements are clear. Assume that the statements hold for $k-1$. In particular we have that $\aff_i(M)$ is a Lie-group.

First we show that
\begin{equation}\label{afftart}
\aff_i(M)^0\subseteq h(\aff_i(N))
\end{equation}

To see this we use that $\aff_i(M)$ is a Lie group and so every element $\alpha\in\aff_i(M)^0$ is connected with the unit element with a continuous path $p:[0,1]\rightarrow\aff_i(M)$ with $p(0)=1$ and $p(1)=\alpha$.
Let $n\in\mathbb{N}$ be sufficiently big and let $\alpha_i=p((i-1)/n)^{-1}p(i/n)$. Then $\alpha=\prod_{i=1}^n\alpha_i$.
Lemma \ref{smalltrans} implies that if $n$ is big enough then for every $\alpha_i$ there is $\beta_i\in\aff_i(N)$ with $h(\beta_i)=\alpha_i$. Let $\beta=\prod_{i=1}^n\beta_i$. We have that $h(\beta)=\alpha$.

The see the first statement we observe that (\ref{afftart}) implies that $h(\aff_i(N))$ is a Lie-group.
It follows from lemma \ref{kerchar2} that $\aff_i(N)$ is an extension of a Lie-group by another Lie-group. Since $\aff_i(N)$ is a Polish group and $\aff_i(N)\cap\ker (h)$ is a closed subgroup we get that $\aff_i(N)$ is a Lie-group (See the appendix of \cite{HKr}).

Now we show the second statement. Since $h$ is continuous we have that $h(\aff_i(N)^0)=h(\aff_i(N))^0$. Equation (\ref{afftart}) implies that $\aff_i(M)^0\subseteq h(\aff_i(N))^0$ and so $\aff_i(M)^0\subseteq h(\aff_i(N)^0)$. The other containment is trivial.
\end{proof}

\begin{corollary} The action $\aff(N)^0$ is transitive on the connected components of $N$.
\end{corollary}

\begin{proof}
By induction $\aff(M)^0$ acts transitively on the connected components of $M$ and furthermore $A_k\subseteq\aff(N)$. By theorem \ref{transurj} $\aff(M)^0=h(\aff(N)^0)$. It follows that the group $T$ generated by $A_k$ and $\aff(N)^0$ is transitive on the connected components of $N$.  Since $A_k^0$ is a finite index subgroup in $A_k$ we have that $\aff(N)^0$ is of finite index in $T$. This is only possible if $\aff(N)^0$ is already transitive on the connected components.
\end{proof}

\begin{definition} A $k$-step nilspace is called {\bf torsion free} if all the structure groups $A_i$ have torsion free dual groups.
\end{definition}

Note that a compact finite dimensional abelian group $A$ has torsion free dual group if and only if $A$ is isomorphic to $(\mathbb{R}/\mathbb{Z})^n$ for some natural number $n$.

\begin{theorem} If $N$ is finite rank torsion free $k$-step nilspace then $N$ is a nilmanifold with structure corresponding to the central series $\{\aff_i(N)^0\}_{i=1}^k$ in $\aff(N)^0$.
\end{theorem}

\begin{proof} We prove the statement by induction on $k$. If $k=1$ then $N$ is an abelian group and the statement is trivial. Assume that it is true for $k-1$. Let $x\in N$ be any fixed point and let $\Gamma$ be the stabilizer of $x$ in $\aff(N)^0$. It is clear from the compactness of $N$ and transitivity of $\aff(N)^0$ that $\Gamma$ is a co-compact subgroup in $\aff(N)^0$. 
It follows that $N$ is a nilmanifold. We have to determine the nilspace structure on $N$.

From theorem \ref{transurj} and our induction hypothesis it follows that for every cube $c\in C^n(N)$ there is a cube $c'\in C^n(N)$ such that $c'$ is translation equivalent with the constant $x$ cube and $\pi_{k-1}(c)=\pi_{k-1}(c')$.
It follows from theorem \ref{bundec} that $c-c'\in C^n(\mathcal{D}_k(A_k))$. Since $A_k\subset\aff_k(N)$ it is easy to that $c$ is translation equivalent with $c'$ with translations from $A_k$.
\end{proof}

\end{document}